\documentclass[11pt]{article}

\usepackage[utf8]{inputenc} 
\usepackage[T1]{fontenc}    
\usepackage{url}            
\usepackage{booktabs}       
\usepackage{nicefrac}       
\usepackage{microtype}      
\usepackage{multirow}
\usepackage{array}
\usepackage[noblocks]{authblk}
\usepackage{amssymb,arydshln}
\usepackage[nodayofweek]{datetime}
\usepackage{float}
\usepackage{bm}
\usepackage{algorithm} 
\usepackage{algpseudocode}

\usepackage{tocloft}  
\usepackage[shortlabels]{enumitem}
\setlist[itemize]{itemindent=0ex,parsep=3pt,itemsep=0pt,leftmargin=\parindent,topsep=5pt,labelwidth=0.8em,labelsep=0.7em}
\setlist[enumerate]{label={\arabic*)},itemindent=0ex,parsep=3pt,itemsep=0pt,leftmargin=\parindent,topsep=5pt,labelwidth=0.9em,labelsep=0.6em}

\usepackage[font=small,labelfont=bf]{caption}
\usepackage{amsmath,amsfonts,nccmath,mathtools,mathrsfs}
\usepackage{tcolorbox}
\usepackage{xcolor}

\usepackage[margin=1in]{geometry}

\usepackage{pgfplots}
\usepgfplotslibrary{fillbetween}
\usetikzlibrary{arrows.meta}
\tikzset{>=latex}

\DeclarePairedDelimiterX\setv[2]{\{}{\}}{#1 \;\delimsize\vert\; #2}

\pgfmathdeclarefunction{cubic}{1}{%
  \pgfmathparse{-2*(#1+2)*(#1+2)*(#1-2)+40}%
}

\pgfmathdeclarefunction{bicuadratic}{1}{%
  \pgfmathparse{(#1-1)*(#1-1)*(#1-1)*(#1-1)+10}%
}

\pgfmathdeclarefunction{cuadratic}{1}{%
    \pgfmathparse{(-(2*#1)^2)+70}%
}

\usepackage{pifont}

\usepackage[numbers,merge,sort&compress]{natbib}
\usepackage{amsthm}
\usepackage{graphicx,color}
\usepackage{subcaption,setspace}

\usepackage[colorlinks = true,
            linkcolor = blue,
            urlcolor  = blue,
            citecolor = blue,
            anchorcolor = blue]{hyperref}

\graphicspath{{figure/}}

\newtheorem{theorem}{Theorem}
\newtheorem{fact}{Fact}
\newtheorem{proposition}{Proposition}
\newtheorem{lemma}{Lemma}
\newtheorem{corollary}{Corollary}

\newtheorem{assumption}{Assumption}

\theoremstyle{definition}
\newtheorem{definition}{Definition}
\newtheorem{example}{Example}
\newtheorem{system}{System}
\newtheorem{remark}{Remark}

\usepackage[nameinlink]{cleveref}
\crefname{equation}{}{}
\crefname{theorem}{Theorem}{Theorems}
\crefname{corollary}{Corollary}{Corollaries}
\crefname{example}{Example}{Examples}
\crefname{assumption}{Assumption}{Assumptions}
\crefname{lemma}{Lemma}{Lemmas}
\crefname{proposition}{Proposition}{Propositions}
\crefname{figure}{Figure}{Figures}
\crefname{table}{Table}{Tables}
\crefname{fact}{Fact}{Facts}
\crefname{system}{System}{Systems}
\crefname{conjecture}{Conjecture}{Conjectures}
\crefname{section}{Section}{Sections}
\crefname{appendix}{Appendix}{Appendices}
\Crefname{equation}{}{}
\Crefname{theorem}{Theorem}{Theorems}
\Crefname{corollary}{Corollary}{Corollaries}
\Crefname{example}{Example}{Examples}
\Crefname{lemma}{Lemma}{Lemma}
\Crefname{proposition}{Proposition}{Proposition}
\Crefname{figure}{Figure}{Figures}
\Crefname{table}{Table}{Tables}
\Crefname{section}{Section}{Sections}
\Crefname{appendix}{Appendix}{Appendices}

\newcommand{\tr}{{{\mathsf T}}}

\newcommand{\her}{{{\mathsf H}}}

\newcommand{\argmin}{\mathop{\mathrm{arg\,min}}\limits}

\newcommand{\Hinf}{{\mathcal{H}_\infty}}

\newcommand{\bT}{{\mathbf{T}}}

\newcommand{\diag}{\mathrm{diag}}

\newcommand{\bigO}{\mathcal{O}}

\newcommand{\RR}{\mathbb{R}} 
\newcommand{\innerproduct}[2]{\left \langle #1, #2 \right\rangle }

\graphicspath{{figure/}}

\usepackage{tikz}
\usepackage{pgfplots}
\pgfplotsset{compat=newest}

\makeatletter
\newcommand{\removelatexerror}{\let\@latex@error\@gobble}
\makeatother

\title{\bf \Large 
Weak Convexity and Proximal Bundle Methods for Nonsmooth Policy Optimization in Robust Control\thanks{This work is supported by NSF CAREER 2340713 and NSF CMMI 2320697. Emails: \texttt{y1watanabe@ucsd.edu}; \texttt{fliao@ucsd.edu}; \texttt{zhengy@ucsd.edu}.}
}

\author[1]{Yuto Watanabe} 
\author[1]{Feng-Yi Liao}
\author[1]{Yang Zheng}
\affil[1]{\small Department of Electrical and Computer Engineering, University of California San Diego \vspace{-2mm}}

\date{ \small \today \vspace{-3mm}} 

\begin{document}

\maketitle
\vspace{-5mm}

\begin{abstract}
We study policy optimization for discrete-time robust $\mathcal{H}_\infty$ control with static output-feedback,
and present the first feasibility-preserving algorithm with a deterministic, non-asymptotic complexity guarantee.
This problem naturally leads to a nonsmooth and nonconvex optimization over the set of stabilizing feedback gains. We first establish several structural properties of the $\mathcal{H}_\infty$ cost. In particular, we show that the cost is weakly convex on every convex subset of a sublevel set. For the state-feedback case, we further establish a weak Polyak--{\L}ojasiewicz inequality, which ensures that every stationary point is globally optimal. Building on these properties, we develop a proximal bundle method for $\mathcal{H}_\infty$ policy optimization. The proposed method can be viewed as an implementable approximation of the proximal point method and uses only function value and subgradient information. We show that all iterates remain stabilizing and establish a deterministic non-asymptotic complexity bound of
$\mathcal{O}(\max\{\eta^{-4},\epsilon^{-2}\})$
for finding an $(\eta,\epsilon)$-stationary point. Numerical experiments illustrate our theoretical results.
\end{abstract}


\section{Introduction}\label{section:introduction}

Robustness against unknown disturbances is a central objective in control system design. Among various robust control formulations, $\mathcal{H}_\infty$ control provides a standard framework for designing a stabilizing controller that minimizes the worst-case amplification from exogenous disturbances to outputs \cite{zhou1996robust}. Classical synthesis methods based on Riccati equations and linear matrix inequalities (LMIs) are well established for the centralized setting \cite{doyle1988state,gahinet1994linear}. These methods often do not optimize over controller parameters directly and rely on additional auxiliary variables. Moreover, they generally require explicit model information and may not preserve prescribed controller structures.

Direct policy optimization offers an alternative approach in which the controller parameters are directly optimized without explicitly introducing auxiliary variables.
This approach often achieves favorable empirical performance and has been successfully applied to practical control problems \cite{apkarian2017h}. 
This approach has recently attracted renewed attention \cite{hu2023toward,talebi2024policy},
as it only uses function value and/or (sub)gradient oracles and is more amenable to model-free design than approaches based on LMIs and Riccati equations. 
It is well-known that $\mathcal{H}_\infty$ synthesis is a constrained nonsmooth nonconvex optimization problem in the policy space \cite{apkarian2006nonsmooth,apkarian2009proximity,guo2022global,tang2023global}. The feasible domain consists of stabilizing feedback gains and is generally nonconvex; in the output-feedback setting, it may even be disconnected. The objective is also nonsmooth because of the worst-case maximization inherent in the $\mathcal{H}_\infty$ norm. 
These features make convergence analysis of local policy search algorithms challenging, and existing works only establish at most asymptotic convergence results \cite{apkarian2006nonsmooth,burke2006hifoo,apkarian2009proximity,burke2005robust} or probabilistic convergence guarantees \cite{guo2022global,guo2023complexity,wang2026zeroth}.
To the best of our knowledge, the deterministic and non-asymptotic convergence rate with feasibility guarantees in $\mathcal{H}_\infty$ policy optimization remains~open. 

In this paper, we develop a new first-order policy optimization algorithm for solving the discrete-time static output-feedback $\mathcal{H}_\infty$ control problem. Our algorithm builds on the class of proximal bundle methods~(PBMs) \cite{lemarechal1981bundle,kiwiel2000efficiency}. By exploiting the structural properties of the problem, we establish the first deterministic, non-asymptotic convergence-rate guarantee for $\mathcal{H}_\infty$ policy optimization. 
We first conduct a detailed structural analysis of the $\mathcal{H}_\infty$ cost 
using complementary frequency-domain and linear matrix inequality (LMI) representations. This analysis establishes the weak convexity of the cost and reveals its hidden convexity.
Building on these properties, we develop a PBM as an implementable approximation of the proximal point method (PPM). While the PPM requires solving an exact proximal subproblem at each iteration and is therefore generally impractical, the PBM replaces this computation with a tractable inner loop while retaining the essential descent property.
Specifically, the inner loop locally convexifies the  $\mathcal{H}_\infty$ cost using its weak convexity and solves a proximal subproblem constructed from a piecewise-affine lower approximation.

More specifically, our main contributions of this work are twofold.

\begin{enumerate}
\item \textbf{Structural properties.} We reveal new weak and hidden convexity properties in $\mathcal{H}_\infty$ policy optimization. 
First, we show that the discrete-time static output-feedback $\mathcal{H}_\infty$ cost is weakly convex on every convex subset of a sublevel set (\cref{theorem:local_wc}). Unlike the local weak convexity guaranteed by the known lower-$C^2$ property \cite{noll2013bundle}, our weak-convexity result is uniform over the prescribed sublevel set, which allows for principled algorithm design and analysis.
{Second, for the state-feedback setting, we establish a weak Polyak--\L{}ojasiewicz (PL) inequality (\cref{theorem:weak-PL}) of the form 
$\mu\bigl(J(K)-J^\star\bigr)
\leq
\operatorname{dist}\bigl(0,\partial J(K)\bigr)$ 
for any $K$ on a sublevel set,
where $\mu>0$ is a uniform constant.} 
This inequality directly ensures the global optimality of every stationary point, recovering \cite[Theorem 1]{guo2022global} as a corollary. Our proof uses a partial minimization for a convex lifting based on the non-strict bounded real lemma \cite{boyd1994linear,rantzer1996kalman}, which extends convex-lifting arguments for differentiable costs \cite{umenberger2022globally,watanabe2026gradient} to the nonsmooth setting. 
Finally, we also provide a Fr\'{e}chet subdifferential formula for the $\mathcal{H}_\infty$ cost (\cref{proposition:subgradient}). 
By exploiting the lower-$C^2$ property, our derivation is simpler than the weak topology argument in \cite[Appendix~B.2]{zheng2023benign} and allows for infinitely many active frequencies, which are not covered by \cite[Section~III]{apkarian2006nonsmooth}.

\item \textbf{Algorithm and convergence guarantees.} We develop a feasibility-preserving PBM with a deterministic, non-asymptotic convergence rate. 
The proposed method extends the PBM in \cite{liang2023proximal,liao2025proximal} from globally weakly convex problems to the nonconvex constrained $\mathcal{H}_\infty$ optimization. 
Our~\cref{theorem:main-result} shows that the total number of iterations to find an $(\eta,\epsilon)$-stationary point (see \Cref{def:inexact-stationary}) is at most 
 $   \mathcal{O}\left(
    \max\left\{
  {\eta^{-4}}, 
    {\epsilon^{-2}} 
    \right\}
    \right). $ 
Setting $\eta=\Theta(\epsilon)$ yields a complexity bound of $\mathcal{O}(\epsilon^{-4})$, matching the state-of-the-art result for globally weakly convex optimization \cite{davis2019stochastic,liang2023proximal,liao2024error,kong2024cost,li2026optimal}. 
Moreover, every point generated in both the inner and outer loops of the PBM remains feasible (\cref{lemma:feasibility_nullstep}). Outer-loop feasibility follows from a quantified descent property inherited from the PPM, and inner-loop feasibility is ensured by a suitable choice of algorithmic parameters, analogous to step-size control in gradient-based methods for the linear-quadratic regulator (LQR) \cite{fazel2018global,hu2023toward}. 
Finally, the PBM can be implemented efficiently using only function value and subgradient oracles.
Each iteration of the inner loop
admits an analytic update under a suitable construction of a lower approximation for the $\mathcal{H}_\infty$~cost.
\end{enumerate}

Some preliminary results appeared in our conference version \cite{watanabe2025policy}, which studies a standard subgradient method under the assumption that all generated iterates remain feasible.
In this work, we develop a new PBM that guarantees the feasibility of every iterate by design. We also provide complete proofs that were omitted in \cite{watanabe2025policy} and present more extensive numerical experiments. 

\subsection{Related works}

\noindent {\bf Robust control.} 
The $\mathcal{H}_\infty$ control problem is one of the most fundamental robust control problems.
It is known that an $\mathcal{H}_\infty$ {suboptimal} unstructured controller can be obtained by solving algebraic Riccati equations \cite{doyle1988state,zhou1996robust}. An $\mathcal{H}_\infty$ optimal controller can be approximately obtained by bisection or LMI-based convex optimization \cite{gahinet1994linear,scherer1997multiobjective,masubuchi1998lmi}. Once structural constraints are imposed, however, the general problem is known to be NP-hard \cite{blondel1997np}. Frequency-domain methods (such as Youla~parameterization \cite{youla1976modern}, system level synthesis \cite{wang2019system}, and input-output parameterization \cite{furieri2019input}) can provide convex approximations for structured control problems. However, these formulations are generally infinite-dimensional, and it is non-trivial to develop effective finite-dimensional computations \cite{zheng2022system,zheng2020equivalence}.

An alternative class of methods is to optimize the controller parameters directly using nonsmooth nonconvex optimization \cite{apkarian2006nonsmooth,burke2005robust,burke2006hifoo,saeki2006fixed,apkarian2009proximity}.
This approach has proved effective in practice and~underlies the solver package \texttt{HIFOO} \cite{burke2006hifoo} and MATLAB functions (\texttt{hinfstruct} and \texttt{systune}) \cite{apkarian2006nonsmooth,apkarian2009proximity}, with successful applications such as a space probe by ESA \cite{apkarian2017h}. However, its non-asymptotic convergence theory is still incomplete and lacks deterministic convergence rate guarantees. Classical methods \cite{apkarian2006nonsmooth,burke2006hifoo,saeki2006fixed,apkarian2009proximity} mainly provide asymptotic guarantees, whereas recent non-asymptotic convergence results \cite{guo2022global,guo2023complexity,wang2026zeroth} are only probabilistic, as they rely on gradient sampling, randomized smoothing, {zeroth-order techniques, respectively.} Our work continues this line of research and provides the first deterministic, non-asymptotic convergence guarantees for discrete-time $\mathcal{H}_\infty$ policy optimization.

\vspace{3pt}

\noindent \textbf{Policy optimization in control.} 
Direct optimization over feedback gains can be dated back to the 1970s
\cite{levine1970determination}.
In the 2000s, the growing interest in structured robust control and large-scale systems led to
the early policy optimization methods \cite{apkarian2006nonsmooth,burke2006hifoo,apkarian2009proximity} for nonsmooth $\mathcal{H}_\infty$ optimization.
More recently, along with the success of reinforcement learning,
\cite{fazel2018global} established a global analysis of policy-gradient methods for LQR,
showing that direct optimization can enjoy global linear convergence despite the nonconvexity. This motivates many follow-up works for LQR \cite{mohammadi2019global,fatkhullin2021optimizing,bu2019lqr,watanabe2025revisiting}.
Related landscape results have since been developed for Linear Quadratic Gaussian (LQG) control \cite{tang2023analysis}, distributed control \cite{furieri2020learning}, LQ dynamic
games \cite{zhang2019policy}, and mixed $\mathcal{H}_2/\mathcal{H}_\infty$ control \cite{zhang2020policy,pai2026policy}; see recent surveys \cite{hu2023toward,talebi2024policy}.

For nonsmooth $\mathcal{H}_\infty$ optimization, 
in addition to the probabilistic algorithms and guarantees in  \cite{guo2022global,guo2023complexity,wang2026zeroth},
the recent works
\cite{zheng2023benign,zheng2024benign,zheng2025ECL} established benign landscape properties for full-state feedback and dynamic output feedback cases, while no algorithm to seek stationary points~was~proposed.

\vspace{3pt}

\noindent \textbf{Weakly convex optimization and proximal bundle methods.} 
Weak convexity is a classical regularity notion that includes convex, smooth nonconvex functions, and a composition of a convex Lipschitz function with a smooth map \cite{davis2019stochastic}. The lower-$C^2$ property is a type of local weak convexity, and has been studied as a class of subsmooth functions \cite{lewis2007nonsmooth}. This was already exploited in nonsmooth $\mathcal{H}_\infty$ optimization, including the proximity-control method \cite{apkarian2009proximity}, but the resulting guarantees were only asymptotic. The work \cite{davis2019stochastic} established the first non-asymptotic convergence rate of subgradient methods for weakly convex functions. Since then, weak convexity has become an active framework for nonsmooth nonconvex optimization with non-asymptotic convergence guarantees \cite{liang2023proximal,kong2024cost,liao2025proximal,li2026optimal}.

The proximal bundle method (PBM)  was first proposed in \cite{lemarechal1981bundle} for minimizing nonsmooth convex functions. Unlike subgradient methods, which typically require diminishing step sizes to guarantee convergence, the PBM can converge to a global minimizer while using a constant proximal parameter.
Early studies focused primarily on establishing asymptotic convergence \cite{kiwiel1990proximity,kiwiel1995approximations}. The first non-asymptotic result appeared in the early 2000s \cite{kiwiel2000efficiency}, and a comprehensive complexity analysis across different function classes and growth conditions has only emerged recently \cite{diaz2023optimal,liang2023proximal}. 
Although the PBM was originally developed for convex optimization, several works have extended it to certain classes of nonconvex problems \cite{hare2010redistributed,hare2009computing,liao2025proximal,liang2023proximal}. In particular,  \cite{hare2010redistributed,hare2009computing} developed PBMs for lower-$C^2$ optimization, and \cite{liang2023proximal,liao2025proximal} proposed PBMs for unconstrained weakly convex optimization with non-asymptotic convergence guarantees. 

In all of these works, the nonconvexity arises solely from the objective function, whereas the feasible set remains convex. In contrast, the $\mathcal{H}_\infty$ policy optimization problem considered in this paper involves both a nonconvex objective and a nonconvex feasible set induced by the stability requirement. This additional nonconvexity requires a more careful algorithmic design and convergence analysis.

\subsection{Paper organization}
The remainder of the paper is organized as follows. \cref{section:problem-statememt} formulates the discrete-time static output-feedback $\mathcal{H}_\infty$ policy optimization problem and reviews the background from nonsmooth optimization. \cref{section:analysis-1} establishes the basic landscape properties, lower-$C^2$ structure, and uniform weak convexity of the cost. \cref{section:analysis-2} derives the explicit subdifferential formula and studies stationary points, including the weak PL inequality in the full-state feedback setting. \cref{section:algorithm} introduces a proximal bundle method and establishes its deterministic non-asymptotic convergence guarantee. \cref{section:simulations} presents numerical experiments, and \cref{section:conclusion} concludes the paper.

\section{Preliminaries and problem statement}\label{section:problem-statememt}

Here, we first formulate the $\Hinf$ policy optimization problem in \cref{subsection:Hinf-formulation}
and then review several basic concepts in nonsmooth nonconvex optimization in \cref{subsection:prelim-nonsmooth-nonconvex-opt}. 

\subsection{Policy optimization for robust $\Hinf$ control}\label{subsection:Hinf-formulation}

Consider the discrete-time linear time-invariant system\footnote{
We focus on discrete-time systems because they are practically relevant and also offer a simple setting for $\Hinf$ optimization, partly due to the coercivity property; see \cite[Lemma 3.2]{guo2023complexity} and \cite[Example 2.3]{zheng2024benign}. See also \Cref{remark:continuous-time}.
}
\begin{equation}\label{eq:dynamic}
\begin{aligned}
x_{t+1} &= A x_t + B u_t + B_w w_t, \qquad
y_t = C x_t,
\end{aligned}
\end{equation}
where $x_t \in \mathbb{R}^{n_x}$ is the state, $u_t \in \mathbb{R}^{n_u}$ is the control input, $y_t \in \mathbb{R}^{n_y}$ is the measured output, and $w_t \in \mathbb{R}^{n_w}$ is the disturbance. The matrices have the dimensions 
$A \in \mathbb{R}^{n_x \times n_x}$,
$B \in \mathbb{R}^{n_x \times n_u}$,
$B_w \in \mathbb{R}^{n_x \times n_w}$, and
$C \in \mathbb{R}^{n_y \times n_x}$.
Throughout, the initial state is fixed at $x_0 = 0$.

We model the disturbance sequence $\mathbf{w}=\{w_0,w_1,\ldots\}$ as an adversarial input with bounded energy. Let $\ell_2^k$ denote the set of square-summable signals in $\mathbb{R}^k$, namely,
\begin{equation*}
\ell_2^k
=
\left\{
\mathbf{w}=\{w_0,w_1,\ldots\}
\;\middle|\;
\|\mathbf{w}\|_2
:=
\sqrt{\textstyle\sum_{t=0}^\infty w_t^\tr w_t}
<\infty
\right\}.
\end{equation*}
The robust $\mathcal H_\infty$ control problem seeks a control sequence
$\mathbf{u}=\{u_0,u_1,\ldots\}$ that minimizes the quadratic cost
$\sum_{t=0}^\infty (x_t^\tr Q x_t + u_t^\tr R u_t)$ against the bounded worst-case disturbance.  
The disturbance bound can be normalized without loss of generality. We thus aim to  minimize
\begin{equation*}
\sup_{\|\mathbf{w}\|_2 \le 1}
\sum_{t=0}^\infty \left(x_t^\tr Q x_t + u_t^\tr R u_t\right)
\end{equation*}
over feedback control policies that generate $\mathbf{u}=\{u_0,u_1,\ldots\}$.
Throughout the paper, we impose the following assumption.
\begin{assumption} \label{assumption:stablizability}
    The system matrices $(A, B)$ are stabilizable, and
    $B_w$ and $C$ are full row rank.
    The weight matrices $Q, R$ are positive definite.
\end{assumption}

As the policy class,
we restrict our attention to the class of linear static output-feedback policies 
\begin{equation}\label{eq:static-policies}
u_t = K y_t, \qquad t = 0,1,2,\ldots,
\end{equation}
and study the policy optimization problem
\begin{equation}\label{eq:robust-policy-optimization}
\begin{aligned}
\min_{K \in \mathbb{R}^{n_u \times n_y}}
\sup_{\|\mathbf{w}\|_2 \le 1}
\quad &
\sum_{t=0}^\infty \left(x_t^\tr Q x_t + u_t^\tr R u_t\right) \\
\text{subject to}\quad &
\text{\cref{eq:dynamic,eq:static-policies}},\;\; x_0=0.
\end{aligned}
\end{equation}
For a given gain $K \in \mathbb{R}^{n_u \times n_y}$, the closed-loop dynamics become
$x_{t+1} = (A+BKC)x_t + B_w w_t$.
We define the stabilizing set by
\begin{equation}
\mathcal K
:=
\left\{
K \in \mathbb{R}^{n_u \times n_y}
\;\middle|\;
\rho(A+BKC) < 1
\right\},
\end{equation}
where $\rho(\cdot)$ denotes the spectral radius. We make another assumption. 

\begin{assumption}\label{assumption:K_nonempty}
    The set $\mathcal{K}$ is non-empty. 
\end{assumption}

If $K \in \mathcal K$, then the corresponding state trajectory
$\mathbf x = \{x_0,x_1,\ldots\}$ belongs to $\ell_2^{n_x}$ for every disturbance $\mathbf w \in \ell_2^{n_w}$. Define the performance output by
$z_t = \begin{bmatrix}
    Q^{1/2}\\
    R^{1/2}KC
\end{bmatrix} x_t$
and
$\mathbf z = \{z_0,z_1,z_2,\ldots\}$.
Then, we also have 
$\|\mathbf z\|_2^2
=
\sum_{t=0}^\infty \left(x_t^\tr Q x_t + u_t^\tr R u_t\right)
< \infty,
\
\forall\, \mathbf x \in \ell_2^{n_x},\;
\mathbf w \in \ell_2^{n_w}.
$
Thus, for each stabilizing gain $K \in \mathcal K$, the closed-loop system in \cref{eq:dynamic,eq:static-policies} defines a bounded linear map from $\ell_2^{n_w}$ to $\ell_2^{n_x+n_u}$. We denote this operator by $\mathbb T_{zw}(K)$ and define its induced $\ell_2$ norm as
\begin{equation}\label{eq:linear-operator-time-domain}
\|\mathbb T_{zw}(K)\|
:=
\sup_{\|\mathbf w\|_2 \ne 0}
\frac{\|\mathbf z\|_2}{\|\mathbf w\|_2},
\qquad
\forall K \in \mathcal K.
\end{equation}
Problem \cref{eq:robust-policy-optimization} is then equivalent to
$\min_{K \in \mathcal K} \|\mathbb T_{zw}(K)\|^2$.
The operator \cref{eq:linear-operator-time-domain} admits the usual frequency-domain expression. Specifically, let $\mathbf T_{zw}(K,\omega)$ denote the transfer matrix from $\mathbf w$ to $\mathbf z$:
\begin{equation} \label{eq:transfer-function-w-z}
\mathbf T_{zw}(K,\omega)
=
\begin{bmatrix}
Q^{1/2} \\
R^{1/2}KC
\end{bmatrix}
\left(e^{j\omega}I-(A+BKC)\right)^{-1} B_w,
\end{equation}
where $\omega \in [0,2\pi]$ is the frequency variable. Its $\mathcal H_\infty$ norm is defined by
\begin{equation}\label{eq:Hinf-norm}
\|\mathbf T_{zw}(K)\|_{\Hinf}
=
\max_{\omega \in [0,2\pi]}
\sigma_{\max}\!\left(\mathbf T_{zw}(K,\omega)\right),
\qquad
\forall K \in \mathcal K,
\end{equation}
where $\sigma_{\max}(\cdot)$ denotes the largest singular value. By \cite[Theorem 4.4]{zhou1996robust}, the induced $\ell_2 \to \ell_2$ norm and the $\mathcal H_\infty$ norm coincide, i.e., 
$\|\mathbb T_{zw}(K)\|
=
\|\mathbf T_{zw}(K)\|_{\Hinf}$,
for 
any $K \in \mathcal K$.

Accordingly, \cref{eq:robust-policy-optimization} can be viewed as the following static-policy $\mathcal H_\infty$ optimization 
\begin{equation}\label{eq:policy-optimization-main}
J^\star = \min_{K \in \mathcal K} J(K),
\end{equation}
where
$J(K) := \|\mathbf T_{zw}(K)\|_{\Hinf}$.
Note that \cref{eq:robust-policy-optimization} minimizes $J(K)^2$, whereas \cref{eq:policy-optimization-main} minimizes $J(K)$. These formulations differ only by a square of the objective and therefore have the same optimal solutions.

\subsection{Preliminaries for nonsmooth nonconvex optimization}
\label{subsection:prelim-nonsmooth-nonconvex-opt}
Problem \cref{eq:policy-optimization-main} is in general nonsmooth and nonconvex.
We here review several basic concepts for nonsmooth and nonconvex optimization.

Consider a locally Lipschitz continuous function $f:  U \to \RR$, where $U \subseteq \RR^n$ is an open set. The \textit{Fr\'echet subdifferential} of $f$ at $\bar x\in U$, denoted as $\partial f(\bar x)$, is defined as:
\begin{align*}
      \partial f(\bar x) = \left \{ v \in \RR^n \mid \liminf_{y \to \bar x,\, y\neq\bar x  }  \frac{ f(y) - f(\bar x) -  \innerproduct{v}{y-\bar x}}{\|y-\bar x\|} \geq 0 \right\}.
\end{align*}
We call an element of the subdifferential $\partial f(\bar x)$ a \textit{subgradient}. If $f$ is convex on $U $,
the subdifferential $\partial f(\bar x)$ admits a more concrete form \cite[Proposition 8.12]{rockafellar2009variational}
\begin{align*}
     \partial f(\bar x) = \left \{ v \in \RR^n \mid f(y) \geq f(\bar x) + \innerproduct{v}{y- \bar x} , \forall y \in U \right\},
\end{align*}
meaning that the linear function $y \mapsto f(\bar x) + \innerproduct{v}{y-\bar x}$ is a global lower bound for $f$ over $U$. If $f$ is differentiable at $\bar x$, then we have $\partial f(\bar x) = \{ \nabla f(\bar x)\}$ \cite[Exercise 8.8]{rockafellar2009variational}.

We say the function $f$ is \textit{$m$-weakly convex} on a convex set $U$ if $x \mapsto f(x) + \frac{m}{2}\|x\|^2$ is convex on $U$. 
The Fr\'echet subdifferential of $m$-weakly convex functions also admits a concrete form: the subdifferential of an $m$-weakly convex function is characterized by the \textit{quadratic} function $y \mapsto f(\bar x)+ \innerproduct{v}{y-\bar x} - \frac{m}{2}\|y-\bar x\|^2$. 
We summarize this in the following result.

\begin{lemma}[{\cite[Lemma 2.1]{davis2019stochastic}}]
    \label{prop:subgrad-weakly}
    Let $U \subseteq \RR^n$ be a convex open set and $f: U \to \RR$ be an $m$-weakly convex function. Given a point $\bar x\in U$, the Fr\'echet subdifferential $\partial f(\bar x)$ can be computed as 
    \begin{align*}
       \partial f(\bar x) = \left \{ v \in \RR^n \mid f(y) \geq f(\bar x) + \innerproduct{v}{y-\bar x}   - \frac{m}{2}\|y-\bar x\|^2, \forall y\in 
       U
       \right\}.
    \end{align*}
\end{lemma}

Due to the nonconvexity,  
the goal 
is often to find a stationary point $x$, i.e., $0 \in \partial f(x)$. However, it is generally impossible to find an exact stationary point in finite time.
Accordingly,
we define $\epsilon$-inexact subdifferential for weakly convex functions as follows.
\begin{definition}
    \label{def:inexact-subdiff}
    Let $f: U \to \RR$ be an $m$-weakly convex function, where $U \subseteq \RR^n$ is a convex open set. Given a point $\bar x\in U$ and a constant $\epsilon \geq 0$, the $\epsilon$-inexact subdifferential of $f$ at $\bar x$, denoted as $\partial_{\epsilon}f(\bar x)$, is defined as
    \begin{align*}
       \partial_{\epsilon} f(\bar x) = \left \{ v \in \RR^n \mid f(y) \geq f(\bar x) + \innerproduct{v}{y-\bar x}   - \frac{m}{2}\|y-\bar x\|^2 - \epsilon, \forall y\in U \right\}.
    \end{align*}    
\end{definition}
Clearly, when $\epsilon = 0$, the subdifferential $\partial_{\epsilon} f(\bar x)$ recovers the true subdifferential $\partial f(\bar x)$. This definition of $\partial_{\epsilon} f(\bar x)$ is similar to the $\epsilon$-subdifferential for a convex function \cite[Definition 1.1.1]{hiriart2013convex}.  
All the definitions above assume the function has a convex domain $U$.  However, in policy optimization for robust control, the domain of the objective function  (i.e., the set $\mathcal{K}$) is naturally nonconvex.
We here introduce a new definition for the $\epsilon$-inexact subdifferential. 
Given an open set $V \subseteq \RR^n$ and a point $\bar x$, we say that the open ball $U \subseteq V$ with its center at $\bar x$ is \textit{the largest open ball} on $V$ if it includes any open ball $W\subseteq V$ with its center at $\bar x$. 

\begin{definition}
    \label{def:inexact-subdiff-redefined}
    Let $V \subseteq \RR^n$ be an open set and $f:V \to \RR$  be $m$-weakly convex on \textit{any} convex subset of $V$. Given a point $\bar x \in V$ and $\epsilon \geq 0$, we define the $\epsilon$-inexact subdifferential of $f$ at $\bar x$ as 
    \begin{align*}
        \partial_{\epsilon} f(\bar x) = \left \{ v \in \RR^n \mid f(y) \geq f(\bar x) + \innerproduct{v}{y-\bar x}   - \frac{m}{2}\|y-\bar x\|^2 - \epsilon, \forall y\in U_{\bar x} \right\},
    \end{align*}
    where $U_{\bar x}$ is the largest open ball around $\bar x$ on $V$. 
\end{definition}
If the domain $V$ is $\mathbb{R}^n$, \Cref{def:inexact-subdiff-redefined} reduces to \Cref{def:inexact-subdiff}. 
We then consider the following $(\eta,\epsilon)$-inexact stationary measure. For $\epsilon = 0$ and $\eta = 0$, the stationary definition below recovers the exact stationary, i.e., $ 0 \in \partial f(\bar x)$.
\begin{definition}[$(\eta,\epsilon)$-stationarity]
    \label{def:inexact-stationary}
    Let $V \subseteq \RR^n$ be an open set and $f:V \to \RR$ be $m$-weakly convex on any convex subset of $V$, and $\eta , \epsilon \geq 0$. A point $\bar x \in V$ is called an $(\eta,\epsilon)$-stationary point if there exists $g \in \partial _{\epsilon} f(\bar x)$ such that $\|g\|\leq \eta$.  
\end{definition}

\section{Basic landscape properties and weak convexity}\label{section:analysis-1}

In this section, we begin by reviewing the basic properties with illustrative examples, and then establish weak convexity of the $\Hinf$ cost. 

\subsection{Basic landscape properties} \label{subsection:basic-properties}

We summarize several basic properties of the $\mathcal{H}_\infty$ optimization problem  \cref{eq:policy-optimization-main} below.

\begin{fact}
\label{lemma:basic-properties}
    Suppose \cref{assumption:stablizability,assumption:K_nonempty} hold. 
    Consider the $\mathcal{H}_\infty$ optimization problem with static linear policies in \cref{eq:policy-optimization-main}. 
    Then the following statements hold. 
    \begin{enumerate}
        \item The stabilizing set $\mathcal{K}$ is open and generally nonconvex. 
        With full-state measurement (i.e., $C=I_{n_x}$), $\mathcal{K}$ is always path-connected, whereas in the general output-feedback case it can be disconnected. 
        \item The $\mathcal{H}_\infty$ cost function $J:\mathcal{K}\to\mathbb{R}$ is, in general, nonsmooth and nonconvex. Furthermore, it is coercive, i.e., 
        $
        J(K)\to\infty \, \text{as } K \to \partial\mathcal{K} \text{ or } \|K\|_F \to \infty.
        $ 
        \item For any sublevel set $\mathcal{K}_\nu := \{K \in \mathcal{K} \mid J(K) \leq \nu\}$, the function $J$ is $\ell_\nu$-Lipschitz continuous for some constant $\ell_\nu > 0$. 
    \end{enumerate}
\end{fact} 

These properties are well known in the literature
\cite{apkarian2006nonsmooth,guo2022global,guo2023complexity,tang2023global,zheng2023benign,apkarian2009proximity}. 
Instead of reproducing the proofs, we briefly discuss the main ideas and provide illustrative examples.
First, it is easy to construct instances in which $\mathcal{K}$ is nonconvex; see, for example, \cref{example:hinf_nonconvexity,example:saddle_local_min} below.
When $C=I_{n_x}$, the path-connectivity of $\mathcal{K}$ follows from an equivalent convex reformulation, since every convex set is path-connected.
Next, $J$ is naturally nonconvex because its domain $\mathcal{K}$ is itself nonconvex. 
Its nonsmoothness comes from two sources in \cref{eq:Hinf-norm}: 
(i) the maximum singular value of a complex matrix and 
(ii) the maximization over the frequency variable.
The coercivity of $J$ can be established by deriving a quadratic lower bound; see \cite[Lemma 3.2]{guo2023complexity}. 
This implies that each sublevel set
\[
\mathcal{K}_\nu
= \{K \in \mathcal{K} \mid J(K) \leq \nu\}
\]
is compact. 
Finally, recall that $J$ is locally Lipschitz continuous on the open set
$\mathcal K$. Since the sublevel set $\mathcal K_\nu \subseteq \mathcal{K}$ is compact, this local Lipschitz property becomes uniform on
$\mathcal K_\nu$ and thus $J$ is Lipschitz continuous on $\mathcal K_\nu$ \cite[Exercise 2.14]{clarke1998nonsmooth}.
As we show later, we can further derive a more explicit Lipschitz constant estimation for any convex subset of $\mathcal{K}_\nu$ using the subdifferential~of~$J$.

We now present two examples illustrating \cref{lemma:basic-properties}.

\begin{figure*}
\begin{subfigure}{.24\textwidth}
  \centering
\includegraphics[width=0.8\columnwidth]{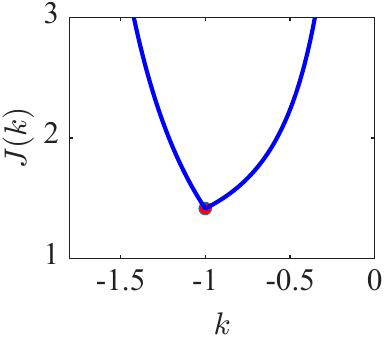}
\caption{$J(k)$ for system \Cref{eq:example-1}}
    \label{fig:1d}
\end{subfigure}%
\begin{subfigure}
{.24\textwidth}
    \centering
\includegraphics[width=0.95\linewidth]{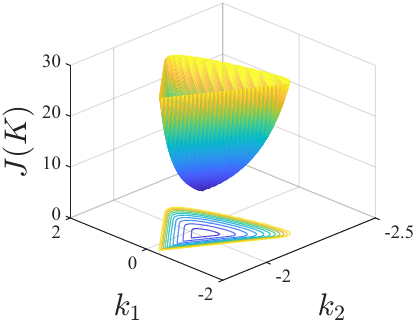}
    \caption{
$J(K)$ for system \Cref{eq:example-2}}
    \label{fig:desouza_xie}
\end{subfigure}
\begin{subfigure}{.24\textwidth}
  \centering
  \includegraphics[width=0.8\columnwidth]{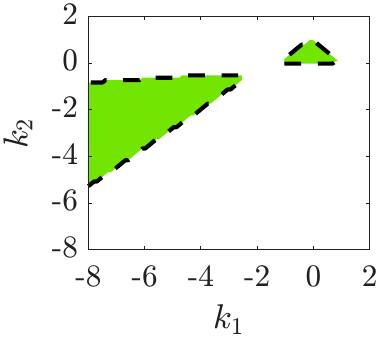}
\caption{
$\mathcal{K}$ in \cref{example:hinf_nonconvexity}}
    \label{fig:outputFB_K}
\end{subfigure}
\begin{subfigure}{.24\textwidth}
  \centering
\includegraphics[width=0.95\columnwidth]{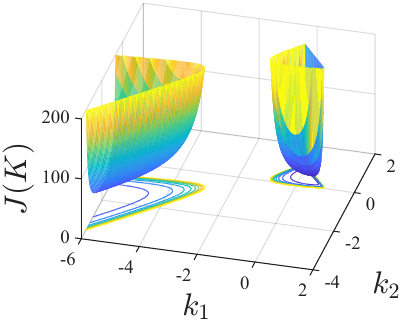}
\caption{
$J(K)$ in \cref{example:hinf_nonconvexity}}
    \label{fig:outputFB_J}
\end{subfigure}%
\caption{
Nonconvex and nonsmooth landscape in $\Hinf$ optimization.
(a)--(b) Nonsmoothness
of the cost functions $J$ in \cref{example:nonsmoothness};
(c)
Disconnectedness of 
$\mathcal{K}$
in \cref{example:hinf_nonconvexity}
with $\alpha=0.13$;
(d)
Nonconvexity of $J$
in \cref{example:hinf_nonconvexity}
with $\alpha=0.13$.
\vspace{-1mm}
}
\label{fig:hinf_plot}
\end{figure*}

\begin{example}[Nonsmoothness of function $J$]\label{example:nonsmoothness}
Consider 
\begin{equation} \label{eq:example-1}
 x_{k+1}  = x_k + u_k +w_k,\ y_k = x_k
 \end{equation}
with $u_k = kx_k$ and $Q=R=1$.
It is clear that $\mathcal{K} = \{k\mid |1+k|<1 \}
=(-2,0).$
A direct computation of its $\mathcal{H}_\infty$ norm gives
\begin{equation} \label{eq:Hinf-example}
    J(k) = 
    \begin{cases}
    {\sqrt{1+k^2}}/{|k|},&k\in[-1,0)\\
    {\sqrt{1+k^2}}/{(k+2)},&k\in(-2,-1). 
    \end{cases}    
\end{equation}
Therefore, $J$ is nonsmooth at $k=-1$, as shown in \cref{fig:1d}.  
In addition, $J(k)\to\infty$ as $k \to 0$ or $k \to -2$, illustrating coercivity.

Next, consider the state-feedback system
\begin{equation} \label{eq:example-2}
    A=\begin{bmatrix}
        0 & 2\\
        4 & 0.2
    \end{bmatrix},\quad
    B = \begin{bmatrix}
        1\\
        0
    \end{bmatrix},
    \quad
    B_w=I_2
\end{equation}
with 
$Q = \diag(1,10^{-3})$ and
$R=1$.
From classical $\mathcal{H}_\infty$ theory, the optimal value of $J(K)$ can be computed as $J^*\approx 8.327$.
The landscape of $J$ for $K=[k_1,k_2]$ is plotted in \cref{fig:desouza_xie}.
This function appears nonsmooth from the contour.
\hfill $\square$
\end{example}

\begin{example}[Nonconvexity of $\mathcal{K}$]\label{example:hinf_nonconvexity}
Consider the following problem data:
\begin{align*}
    &A=
    \begin{bmatrix}
       1 -\alpha & -0.1 & -0.1 \\
0.1 & 1& 0 \\
0 & 0.1 & 1
    \end{bmatrix}
,\; B=
\begin{bmatrix}
   0.1 \\
0 \\
0
\end{bmatrix},\;
B_w = I_3,
\; C=
\begin{bmatrix}
    0 & 1 & 1 \\
1 & -1 & 1
\end{bmatrix},\;
Q = 0.01I_3,\; R = 0.01.
\end{align*}
The static output-feedback gain has the form $K = [k_1,k_2]$.
We can verify numerically that if $0.05 < \alpha \leq 0.13$,
the feasible region $\mathcal{K}$ is disconnected,
which is naturally nonconvex.
For $\alpha=0.13$, the set $\mathcal{K}$ and the corresponding cost $J$ are plotted in \cref{fig:outputFB_K} and \cref{fig:outputFB_J}, respectively. \hfill $\square$
\end{example}

\begin{remark}[$\mathcal{H}_\infty$ control in continuous time]\label{remark:continuous-time}
The continuous-time counterpart of \cref{eq:policy-optimization-main} is also nonsmooth and nonconvex. 
The stabilizing set of gains remains nonconvex; it is path-connected when $C=I_{n_x}$ and may be path-disconnected in general. 
A key difference from the discrete-time setting is that coercivity of $J$ need not hold in continuous time; see \cite[Example~2.3]{zheng2024benign}. 
Consequently, sublevel sets may fail to be compact, and it is not clear whether the continuous-time $\mathcal{H}_\infty$ cost $J$ admits a uniform $\ell_\nu$-Lipschitz constant on such sets. \hfill $\square$
\end{remark}

\subsection{Lower-$C^2$ and weak convexity} \label{subsection:weak-cvx}

We here establish another key technical property: although the $\mathcal{H}_\infty$ cost $J$ is nonconvex, it nevertheless exhibits weak convexity. We will give a precise statement in \Cref{theorem:local_wc} later. 

Recall that a function $f:\mathcal{D}\to\mathbb{R}$ is of class $C^2$ if it is twice differentiable on its open domain $\mathcal{D}$ and both $\nabla f$ and $\nabla^2 f$ are continuous.
Informally, a lower-$C^2$ function may be nonsmooth, but its nonsmoothness arises only through a local maximum of $C^2$ functions, which makes it far more structured than a generic nonsmooth function. 
The formal definition is as follows.

\begin{definition}[Lower-$C^2$ \cite{rockafellar2009variational}] \label{definition:lower-c2}
    We say a function $f:\mathcal{D}\to\RR$, defined on an open domain  $\mathcal{D} \subset \RR^n$, 
is lower-$C^2$, if 
for each $\bar{x}\in\mathcal{D}$, there exist an open neighborhood $V \subset \mathcal{D}$ of $\bar{x}$, a compact set $\mathcal{T}$, and a mapping 
$\varphi:\mathcal{T}\times V \to \mathbb{R}$ such that  
\begin{equation} \label{eq:lower-c2-definition}
    f (x) = \max_{t\in \mathcal{T}} \varphi(t,x),\quad
    \forall x\in V
\end{equation}
where, for each $t\in \mathcal{T}$, the function $\varphi(t,\cdot)$ is $C^2$ in $x$, 
and the mappings $\varphi(\cdot,\cdot)$, $\nabla_x \varphi(\cdot,\cdot)$, 
and $\nabla_x^2 \varphi(\cdot,\cdot)$ 
are jointly continuous on $(t,x)\in \mathcal{T}\times V$. 
\end{definition}

Note that the family of smooth functions $\varphi(t,\cdot)$ in \cref{eq:lower-c2-definition} may depend on the reference point $\bar{x}$.
Thus, a lower-$C^2$ function admits only a \emph{local} representation as a maximum of smooth $C^2$ functions.
Let us provide several simple examples. 
First,
any quadratic function is clearly lower-$C^2$, and so is the maximum of finitely many quadratics\footnote{
In this case, we regard $\mathcal{T}$
as a subset of a discrete topological space.}, e.g.,
$ f(x) = \max_{t\in \{1,\ldots,p\}} f_t(x)$.
The second example is
the maximum eigenvalue function of a symmetric matrix:
\begin{equation} \label{eq:example-eigenvalue}
\lambda_{\max}(X) = \max_{u^\tr u = 1} u^\tr X u,
\end{equation}
where each component $\varphi(u,X)=u^\tr X u$ is linear in $X$, and hence $C^2$.
Likewise, the maximum singular value of a complex matrix $Y \in \mathbb{C}^{p \times m}$ is lower-$C^2$, because
\begin{equation} \label{eq:example-signular-value}
\sigma_{\max}(Y) = \max_{\substack{u \in \mathbb{C}^{m},\, v \in \mathbb{C}^p \\ \|u\|=\|v\|=1}} \mathrm{Re}({v^\her}  Y u),
\end{equation}
where $\mathrm{Re}(\cdot)$ denotes the real part and each component
$\varphi(u,v,Y) = \mathrm{Re}({v^\her} Y u)$
is linear, and hence $C^2$, in $Y$.
These examples show that many spectral functions naturally belong to the lower-$C^2$ class, which is precisely the structure we exploit for the $\mathcal{H}_\infty$ cost; see \cref{eq:Hinf-norm}.

We now show that the $\mathcal{H}_\infty$ cost also falls into this class of lower-$C^2$ functions.

\begin{lemma}\label{fact:lower_C2}
Suppose \cref{assumption:stablizability,assumption:K_nonempty} hold. 
Consider the $\mathcal{H}_\infty$ cost function $J:\mathcal{K}\to\mathbb{R}$ defined in \cref{eq:policy-optimization-main}. 
Then the following properties hold:
\begin{enumerate}
    \item $J$ is a lower-$C^2$ function. 
    \item $J$ is \emph{locally weakly convex}: for every $K_0 \in \mathcal{K}$, there exist constants $m>0$ and $r>0$ with $\mathbb{B}_r(K_0):= \{K \mid \|K - K_0\|_F \leq r\} \subset\mathcal K$ such that 
    \begin{equation} \label{eq:local-weak-convexity}
        K \;\mapsto\; J(K) + \tfrac{m}{2}\|K\|_F^2
    \end{equation}
    is convex over 
    \(\mathbb{B}_r(K_0)\).
    \item $J$ is \emph{subdifferentially regular}, i.e., its directional derivative coincides with its Clarke directional derivative at every point. 
\end{enumerate}
\end{lemma}

\begin{proof}
Similar to \cref{eq:example-signular-value}, the $\mathcal{H}_\infty$ cost in \cref{eq:Hinf-norm} can be written as
\begin{equation*}
J(K)=\max _{\omega \in [0,2\pi]} \max_{\substack{u \in \mathbb{C}^{n_w},\, v \in \mathbb{C}^{n_x+n_u} \\ \|u\|=\|v\|=1}} \mathrm{Re}\left[ v^\her \bT_{zw}(K,\omega)u\right].
\end{equation*} 
For notational convenience, let us define
\begin{equation} \label{eq:local-smooth-hinf}
    \varphi(\theta,K) := \mathrm{Re}\!\left[ v^\her \bT_{zw}(K,\omega) u \right],
\end{equation}
where $\theta = (\omega,u,v) \in \Theta$ and
\begin{equation} \label{eq:Theta}
\Theta := [0,2\pi] \times \{u \in \mathbb{C}^{n_w} : \|u\|=1\} \times \{v \in \mathbb{C}^{n_x+n_u} : \|v\|=1\}.
\end{equation}
We can now write 
\begin{equation} \label{eq:hinf-smooth-componenet}
J(K) = \max_{\theta \in \Theta} \varphi(\theta,K), \quad \forall K \in \mathcal{K}.
\end{equation}
From its definition \cref{eq:local-smooth-hinf}, $\varphi(\theta,\cdot)$ is analytic in $K$ over the open set $\mathcal{K}$, and 
$\varphi$, $\nabla_K\varphi$, and $\nabla^2_K \varphi$ are
jointly continuous in $(\theta,K)$ on 
$\Theta \times \mathcal{K}$. 
We also know that $\Theta$ is compact by definition.
Therefore, by \Cref{definition:lower-c2}, $J$ is lower-$C^2$.
Its weak convexity in the second statement and the subdifferential regularity in the third statement  then follow from
\cite[Theorems 10.31 and 10.33]{rockafellar2009variational}.
\end{proof}

The above argument is inspired by the continuous-time analysis in \cite[Lemma 9]{noll2013bundle}.
Observe that local weak convexity \cref{eq:local-weak-convexity} is a generic consequence of the lower-$C^2$ property \cite[Theorem 10.33]{rockafellar2009variational}. 
However, the corresponding weak convexity constant $m$ may depend on both the point and the neighborhood, and is therefore difficult to quantify explicitly.

It is important to note that the smooth components in
\cref{eq:example-eigenvalue,eq:example-signular-value,eq:hinf-smooth-componenet}
are \emph{global}, in the sense that they do not depend on the reference point $K$.
This is substantially stronger than the general lower-$C^2$ representation in \Cref{definition:lower-c2}.
As a consequence, the $\mathcal{H}_\infty$ cost enjoys a stronger property: 
on any convex subset of a sublevel set, we can obtain a uniform weak convexity constant.

\begin{theorem}[Weak convexity]\label{theorem:local_wc}
Suppose \cref{assumption:stablizability,assumption:K_nonempty} hold. Consider the $\mathcal{H}_\infty$ cost function $J:\mathcal{K}\to\mathbb{R}$ defined in \cref{eq:policy-optimization-main}. Let $\nu > 0$ and define a sublevel set $\mathcal{K}_\nu := \{K \in \mathcal{K} \mid J(K) \leq \nu\}\neq \emptyset$. Then, there exists a constant $m_\nu>0$ such that the function 
\[
K \;\mapsto\; J(K) + \tfrac{m_\nu}{2}\|K\|_F^2
\]
is convex over any nonempty convex subset $V \subset \mathcal{K}_\nu$. 
\end{theorem}

\begin{proof}
We use the equivalent representation of the $\mathcal{H}_\infty$ cost in \cref{eq:hinf-smooth-componenet}. 
For each fixed $\theta\in\Theta$, the mapping
$K\mapsto \varphi(\theta,K)$ in \cref{eq:local-smooth-hinf} is twice continuously differentiable on
$\mathcal K$. Furthermore, the Hessian mapping
$
    (\theta,K)\mapsto \nabla_K^2 \varphi(\theta,K)
$ 
is jointly continuous on $\Theta\times \mathcal K_\nu$. Hence, by compactness of $\Theta$ and $K_\nu$ from \cref{lemma:basic-properties}, 
there exists a constant $m_\nu>0$ such that
\[
    \nabla_K^2\varphi(\theta,K) + m_\nu I \succeq 0,
    \qquad
    \forall (\theta,K)\in \Theta\times\mathcal K_\nu.
\]
Therefore, for every $\theta\in\Theta$, the function
$
    K\mapsto \varphi(\theta,K)+\frac{m_\nu}{2}\|K\|_F^2
$ 
is convex on any convex subset $V\subset\mathcal K_\nu$. Consequently,
\[
\begin{aligned}
    J(K)+\frac{m_\nu}{2}\|K\|_F^2
    &=
    \max_{\theta\in\Theta}
    \left\{
        \varphi(\theta,K)+\frac{m_\nu}{2}\|K\|_F^2
    \right\}
\end{aligned}
\]
is the pointwise maximum of convex functions on $V$, and is therefore convex
on $V$.
\end{proof}

Two key ingredients in the proof above are 1) the smooth components $\varphi(\theta,K)$ in $\mathcal{H}_\infty$ cost representation \cref{eq:hinf-smooth-componenet} are defined uniformly for  all $K \in \mathcal{K}$, and 2) the sublevel subset $\mathcal{K}_\nu$ is compact. These two facts allow us to uniformly lower bound the Hessian of $\varphi(\theta,K)$ over the compact set $\Theta\times\mathcal{K}_\nu$. This directly yields a uniform weak convexity parameter $m_\nu$ on
$\mathcal{K}_\nu$.

\begin{figure}
    \centering    \includegraphics[width=0.7\linewidth]{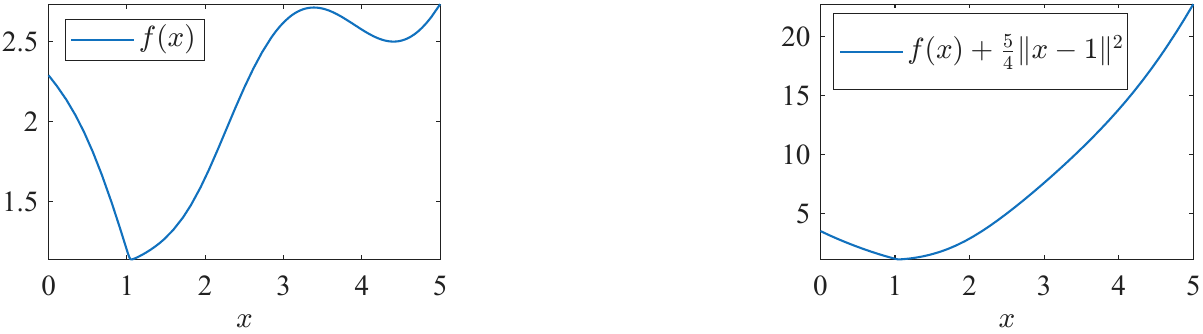}
\caption{
Plots of 
$f$ 
and $f+\frac{5}{4}\|\cdot-1\|^2$ 
on $V=[0,5]$
in
\cref{example:weak_convexity}.}
    \label{fig:wcxv}
\end{figure}

\begin{example}[Weak convexity]\label{example:weak_convexity}
Consider a nonconvex and nonsmooth function 
$f:(-1,\infty)\to\mathbb{R}_+$:
\begin{equation*}
    f(x)
    = \sigma_\mathrm{max}\left(
    c(x)
    \right),\quad
    c(x)=\begin{bmatrix}
        1 & x/2\\
        (1+x)^{-1} & -2\cos
        x
    \end{bmatrix}.
\end{equation*}
This example is not exactly the $\mathcal{H}_\infty$ cost in \cref{eq:Hinf-norm}, since there is no frequency variable, but it captures the essential structure of composing $\sigma_{\max}$ with a smooth mapping.
Indeed, on any compact interval $V\subset(-1,\infty)$, the map $c(\cdot)$ is $L_V$-smooth for some $L_V>0$, and hence $f(\cdot)$ is $m_V$-weakly convex on $V$ for some $m_V>0$.
To illustrate this, we plot $f(x)$ and
$f(x)+\frac{5}{4}\|x-1\|^2$
on $V=[0,5]$ in \cref{fig:wcxv}.
One can clearly see that the added quadratic term convexifies the function over this interval.
\hfill $\square$
\end{example}

\section{Subgradients and stationary points}\label{section:analysis-2}

In this section, we first characterize the Fr\'echet subdifferential of the $\Hinf$ cost. We then discuss its stationary points, and establish a weak PL inequality for the state-feedback case.

\subsection{Subgradient characterization and Lipschitz continuity} \label{subsection:subdifferential}

Recall from \cref{eq:hinf-smooth-componenet} that the $\mathcal{H}_\infty$ cost function admits a
representation as the pointwise maximum of a family of smooth functions.
We can thus use the
standard subdifferential formula for such maximum functions
\cite[Theorem 10.31]{rockafellar2009variational}.

To illustrate the idea, consider the finite pointwise maximum case
$
    f(x)=\max\{f_1(x),\ldots,f_N(x)\},
$ 
where each $f_i$ is continuously differentiable. Let
$
    I(x):=\{i\in\{1,\ldots,N\}\mid f_i(x)=f(x)\}
$ 
denote the active index set at $x$. If there is a unique active function, say
$I(x)=\{i\}$, then $f$ is differentiable at $x$ and
$
    \partial f(x)=\{\nabla f_i(x)\}.
$ 
If multiple functions are active, then the subdifferential of $f$ is described by the convex hull of the gradients of the active functions:
$
    \partial f(x)
    =
    \operatorname{conv}\{\nabla f_i(x)\mid i\in I(x)\}.
$ 
Thus, nonsmoothness arises precisely at points where several smooth components are active.

The same principle applies to the $\mathcal{H}_\infty$ cost. From \cref{eq:hinf-smooth-componenet}, we have 
$
    J(K)=\max_{\theta\in\Theta} \varphi(\theta,K),
$ 
where $\Theta$ is compact and each $\varphi(\theta,\cdot)$ is smooth in $K$. Define the
active index set
\begin{equation} \label{eq:Theta_K-discrete-time}
    \Theta_0(K)
    :=
    \{\theta\in\Theta\mid \varphi(\theta,K)=J(K)\}.
\end{equation}
Then, by \cite[Theorem 10.31]{rockafellar2009variational}, we have 
\begin{equation} \label{eq:J-subdifferential}
    \partial J(K)
    =
    \operatorname{conv}
    \left\{
       \nabla_K \varphi(\theta,K)
        \mid
        \theta\in\Theta_0(K)
    \right\}. 
\end{equation}
In other words, the subdifferential of $J$ is the convex hull of the gradients
of the active smooth components. These active components may correspond to
different maximizing frequencies $\omega$, or to different singular-vector
pairs associated with the same maximizing frequency. 
By explicitly computing the gradient $\nabla_K \varphi(\theta,K)$, we get the following characterization.

\begin{proposition}[Subdifferential characterization]\label{proposition:subgradient}
Suppose \cref{assumption:stablizability,assumption:K_nonempty} hold.
The Fr\'echet subdifferential of $J(\cdot)$ at $K\in\mathcal{K}$ 
is given by \cref{eq:J-subdifferential}, 
where 
\begin{equation} \label{eq:gradient-smooth-component}
\nabla_K \varphi(\theta,K) = 
\mathrm{Re}\left[
C\Gamma_{(K,\omega)} B_wuv^\her
\left(
\begin{bmatrix}
    0\\
    R^{1/2}
\end{bmatrix} 
+
\begin{bmatrix}
Q^{1/2}\\
R^{1/2}KC
\end{bmatrix}
\Gamma_{(K,\omega)} B
\right)\right]^\tr
\end{equation}
and 
$\Gamma_{(K,\omega)} = (e^{j\omega}I -(A+BKC))^{-1}$. 
Here, $\theta=(\omega,u,v)\in\Theta$, with $\Theta$ defined in
\cref{eq:Theta}; the active index set $\Theta_0(K)$ is defined in
\cref{eq:Theta_K-discrete-time}; and the smooth component
$\varphi(\theta,K)$ is defined in \cref{eq:local-smooth-hinf}. 
\end{proposition}

The full proof is given in  \cref{subsection:proof-subgradient}.
To the best of our knowledge, an explicit formula of this form is not readily available for
the discrete-time setting, although continuous-time counterparts have appeared
in \cite{apkarian2006nonsmooth,zheng2023benign}. Importantly, this result also
allows the set of maximizing frequencies $\arg\max_{\omega\in[0,2\pi]}
    \sigma_{\max}\bigl(\bT_{zw}(K,\omega)\bigr)$ to be infinite.  

We use the following example from \cref{eq:example-1} to illustrate \cref{proposition:subgradient}. 

\begin{example}[Subdifferential computation]\label{example:subdifferential}
Consider the scalar system again in \cref{eq:example-1}.
From the closed-form expression \cref{eq:Hinf-example}, it is not difficult to verify that
\begin{equation}\label{eq:subgrad_closed_form}
    \partial J(k) = 
\begin{cases}
\left\{\dfrac{1}{k^2\sqrt{1+k^2}}\right\}, 
& k\in(-1,0), \\
\left[-\dfrac{3}{\sqrt{2}},\dfrac{1}{\sqrt{2}}\right],
& k=-1, \\
\left\{\dfrac{2k-1}{(k+2)^2\sqrt{1+k^2}}\right\},
& k\in(-2,-1).
\end{cases}
\end{equation}
The only nonsmooth point is $k=-1$. We now show that the same expression can
be recovered from \cref{proposition:subgradient}.
For this scalar system, we have
$   \Gamma_{k,\omega}=(e^{j\omega}-1-k)^{-1},
    \, 
    \mathbf{T}_{zw}(k,\omega)=\Gamma_{k,\omega}
    \begin{bmatrix}
        1\\ k
    \end{bmatrix}.
$ 
The maximizing frequencies are $\omega=0$ for $k\in(-1,0)$, 
$\omega=\pi$ for $k\in(-2,-1)$, and every $\omega\in[0,2\pi]$ when
$k=-1$. Note that the maximizing frequencies are infinitely many at $k = -1$.  
A convenient choice of active set representation is
\[
\Theta_0(k)
=
\begin{cases}
\{(0,s,sv_1) \mid s^2=1,\ s\in\mathbb R\},
& k\in(-1,0), \\
\{(\omega,s,sv_2(\omega))\mid s^2=1,\ s\in\mathbb R,\ 
\omega\in[0,2\pi]\},
& k=-1, \\
\{(\pi,s,sv_3)\mid s^2=1,\ s\in\mathbb R\},
& k\in(-2,-1),
\end{cases}
\]
where
$    v_1=\frac{1}{\sqrt{1+k^2}}
    \begin{bmatrix}
        1\\ k
    \end{bmatrix},
    \,
    v_2(\omega)=\frac{e^{-j\omega}}{\sqrt{2}}
    \begin{bmatrix}
        1\\ -1
    \end{bmatrix},
    \, 
    v_3=-v_1.
$ 
Using \cref{proposition:subgradient} 
we obtain
\begin{equation}\label{eq:subgrad_closed_form2}
    \partial J(k) = \operatorname{conv}
    \left\{
        \nabla_k \varphi(\theta,k) 
        \mid
        \theta\in\Theta_0(k)
    \right\}.
\end{equation}
A direct computation gives 
\begin{align*}
    \nabla_k \varphi(\theta,k)
    =
    \mathrm{Re}
    \left[
    \Gamma_{(k,\omega)}
    uv^\her
    \left(
    \begin{bmatrix}
        0\\
        1
    \end{bmatrix}
    + 
    \begin{bmatrix}
        1\\
        k
    \end{bmatrix}\Gamma_{(k,\omega)}
    \right)
    \right]^\tr
    =
    \begin{cases}
        \dfrac{1}{k^2\sqrt{1+k^2}} ,& k\in(-1,0), \\
        \sqrt{2}\cos \omega -\frac{1}{\sqrt{2}},
        & 
        k= -1,\, \omega\in[0,2\pi],\\
        \dfrac{2k-1}{(k+2)^2\sqrt{1+k^2}},
    & k \in (-2,-1).
    \end{cases}
\end{align*}
Since $-\frac{3}{\sqrt{2}}
\leq
\sqrt{2}\cos \omega -\frac{1}{\sqrt{2}}
\leq
\frac{1}{\sqrt{2}}$ for $\omega\in [0,2\pi]$,
the representation \cref{eq:subgrad_closed_form2} coincides with \cref{eq:subgrad_closed_form}.
\hfill $\square$
\end{example}

Note that the subdifferential $\partial J(K)$ is bounded for every $K\in\mathcal K$. Indeed, by \cref{eq:Theta_K-discrete-time}, the active index set $\Theta_0(K)$ is a subset of the compact set $\Theta$, and $\nabla_K\varphi(\theta,K)$ is continuous in $\theta$. Moreover, by the compactness of the sublevel set $\mathcal K_\nu$ and the joint continuity of $\nabla_K\varphi(\theta,K)$ in $(\theta,K)$, these gradients are uniformly bounded over $\Theta\times\mathcal K_\nu$. Consequently, the subdifferential of $J$ is uniformly bounded on $\mathcal K_\nu$. Then, it follows by the standard argument that $J$ is $L_\nu$-Lipschitz continuous on any convex
subset of $\mathcal K_\nu$. The following corollary is immediate. 

\begin{corollary}[Lipschitz continuity]\label{corollary:Lipschitz}
Suppose \cref{assumption:stablizability,assumption:K_nonempty} hold.
 Let
$\mathcal K_\nu:=\{K\in\mathcal K\mid J(K)\le \nu\}\neq\emptyset$. The function $J$ is $\ell_\nu$-Lipschitz over any convex subset of $\mathcal{K}_\nu$, where $\ell_\nu>0$ can be chosen such~that  
\begin{equation*}
    \ell_\nu
    \geq  \max_{K\in\mathcal{K}_\nu,\,\theta \in \Theta}
   \left \| \nabla_K \varphi(\theta,K)
\right 
\|_F, 
\end{equation*}
where $\nabla_K \varphi(\theta,K)$ is given in \cref{eq:gradient-smooth-component}.
\end{corollary}
This corollary gives a Lipschitz estimate on any convex subset
of $\mathcal K_\nu$ directly from the gradient bound. We note that $J$ is also Lipschitz continuous on the entire sublevel set $\mathcal K_\nu$, as mentioned in \cref{lemma:basic-properties}. However, the Lipschitz constant on $\mathcal K_\nu$ could be larger than that in \cref{corollary:Lipschitz} since $\mathcal K_\nu$ is compact but in general nonconvex.

\subsection{Non-optimal stationary points and a weak PL condition}\label{subsection:stationary points}

In this subsection, we study the stationary points of the $\Hinf$ cost $J$.
As is common in nonconvex optimization, the general output-feedback $\mathcal{H}_\infty$ problem may possess non-optimal stationary points, including non-global local minima and saddle points.

\begin{fact}
Suppose \cref{assumption:stablizability,assumption:K_nonempty} hold. 
Consider the $\mathcal{H}_\infty$ cost function $J:\mathcal{K}\to\mathbb{R}$ defined in \cref{eq:policy-optimization-main}.  
For the general static output feedback, i.e., when $C \neq I_{n_x}$, the function $J$ may admit non-optimal stationary points, including non-global local minima and saddle points.
\end{fact}

We illustrate this through \cref{example:saddle_local_min} below.
This example uses the same dynamical system as in \cref{example:hinf_nonconvexity}, but with a different value of $\alpha$.
It shows that such non-optimal stationary points arise naturally from the nonconvexity of the problem and the possible disconnectedness of $\mathcal{K}$.

\begin{example}[Saddle points]\label{example:saddle_local_min}
Consider the problem data in \cref{example:hinf_nonconvexity}.
When $\alpha=0.13$, the feasible set $\mathcal{K}$ has two disconnected components; see \cref{fig:outputFB_K}. 
When $\alpha=0.14$, these components merge and $\mathcal{K}$ becomes connected, as shown in \cref{fig:saddle_K}. 
Further, 
\cref{fig:saddle_J,fig:saddle_J_zoom} show that
the landscape of $J$ exhibits multiple stationary points, including a saddle point.
At
$
K_{\rm s}=[k_1,k_2]\approx[-1.92,-0.26],
$ 
we obtain $J(K_{\rm s})\approx106.72$ and
$\|\nabla J(K_{\rm s})\|_F\le 7.11\times10^{-10}$.
A numerical evaluation of Hessian $H$ gives
eigenvalues approximately $\lambda_\mathrm{min}(H)\approx-174.41$ and $\lambda_\mathrm{max}(H)\approx 2001.75$, showing that
the $\mathcal{H}_\infty$ cost may have a strict saddle.
\hfill $\square$
\end{example}

\begin{figure}
\begin{subfigure}{0.33\columnwidth}
  \centering
\includegraphics[width=0.73\columnwidth]{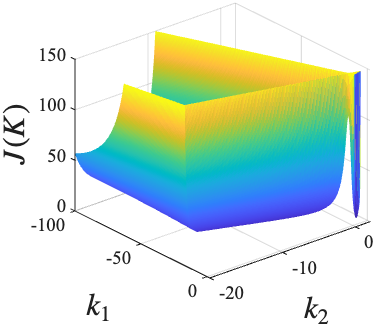}
\caption{
$J(K)$ in \cref{example:saddle_local_min}}
\label{fig:saddle_J}
\end{subfigure}%
\begin{subfigure}{0.33\columnwidth}
  \centering
\includegraphics[width=0.73\columnwidth]{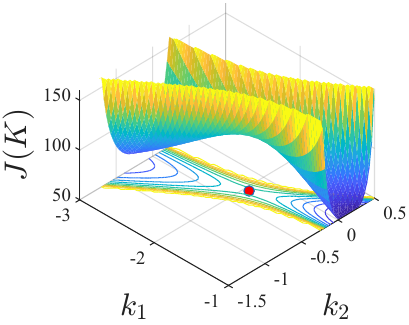}
\caption{
$J(K)$ in \cref{example:saddle_local_min} (zoomed)}
\label{fig:saddle_J_zoom}
\end{subfigure}
\begin{subfigure}{0.33\columnwidth}
  \centering
\includegraphics[width=0.6\columnwidth]{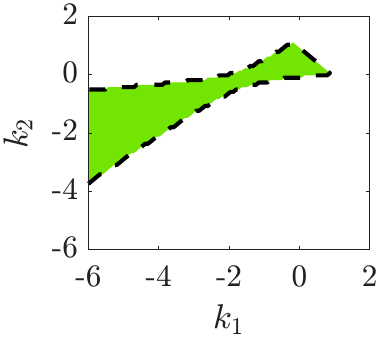}
\caption{$\mathcal{K}$
in \cref{example:saddle_local_min}
}
\label{fig:saddle_K}
\end{subfigure}%
\caption{
Non-optimal stationary points of $J$
with the feasible set $\mathcal{K}$
in the static output-feedback case of \cref{example:saddle_local_min}
(the same system as \cref{example:hinf_nonconvexity} with $\alpha=0.14$).
The red dot indicates a saddle point, and the landscape also contains a local minimum.
}
\label{fig:saddle_point}
\end{figure}

In contrast, under full-state measurement ($C=I_{n_x}$), the $\mathcal{H}_\infty$ optimization problem has a significantly better landscape, with a connected feasible region and a convex-like geometry, as suggested by \cref{fig:1d,fig:desouza_xie}.
Indeed, we can establish a weak PL condition in this case, which implies that any stationary point is globally optimal.

{\begin{theorem}[Weak PL condition]\label{theorem:weak-PL}
Suppose \cref{assumption:stablizability,assumption:K_nonempty} hold.
Assume full-state measurement is available, i.e., $C=I_{n_x}$.
Then, for every $\nu>J^\star$,
there exists a positive constant $\mu>0$ such that
\begin{equation}\label{eq:weakPL}
    \mu\left(J(K)-J^{\star}\right)
    \leq \mathrm{dist}(0,\partial J(K) ),\qquad
    \forall K\in \mathcal{K}_\nu,
\end{equation}
where $\partial J$ denotes the 
Fr\'echet subdifferential
of $J$ and
$J^\star = \min_{K \in \mathcal{K}} J(K)$.
\end{theorem}}

We detail the proof procedure in \cref{subsubsection:proof_wcvx_wpl}, where we also provide an explicit expression for the constant $\mu$ in \cref{eq:weakPL}. 
This property is closely tied to the classical LMI reformulation obtained from the bounded real lemma \cite{boyd1994linear,rantzer1996kalman}, which plays a central role in the proof.
Our argument is motivated by the recent \texttt{ECL} framework \cite{zheng2023benign} together with the result of \cite{umenberger2022globally,watanabe2026gradient}.
Finally, we note that \cref{theorem:weak-PL} immediately ensures that every stationary point is globally optimal. 

\begin{corollary}[Global optimality of stationary points]\label{corollary:stationary_global_optimality}
Under the same setting in \cref{theorem:weak-PL}, we have
$
0\in \partial J(K)
\Leftrightarrow 
K\in {\arg\min}_{K\in\mathcal{K}}J(K).
$
\end{corollary}

\begin{proof}
If $0 \in \partial J(K)$, then
$\mathrm{dist}(0,\partial J(K)) = 0$,
and thus \cref{eq:weakPL} implies $J(K)=J^\star$.
Conversely, if $K\in {\arg\min}_{K\in\mathcal{K}}J(K)$, then $K$ must be stationary, and thus
$0 \in \partial J(K)$ by \cite[Theorem 10.1]{rockafellar2009variational}.
\end{proof}

\subsection{Proof of the weak PL condition in \cref{theorem:weak-PL}}\label{subsubsection:proof_wcvx_wpl}

We here provide a proof of \cref{theorem:weak-PL} by exploiting 
a favorable \textit{partial minimization} 
and \textit{hidden convex} structure of the cost $J$.
Here, we assume that $K\in\mathcal{K}$ is not optimal, i.e., $J(K)>J^\star$. (For $K=K^\star$, the weak PL condition is obvious.)

By the celebrated bounded real lemma, 
it is known that
the problem \cref{eq:policy-optimization-main} with $C=I_{n_x}$ admits a useful reformulation with the additional Lyapunov variable $P$ and a scalar $\gamma$. 
We first recall the classical non-strict bounded real lemma \cite[Theorem 2]{rantzer1996kalman}.

{
\begin{proposition}[Non-strict bounded real lemma]\label{proposition:BRL}
Suppose \cref{assumption:stablizability,assumption:K_nonempty} hold.
Assume $C=I_{n_x}$.
Then, given $\gamma>0$, the following statements
are equivalent.
\begin{enumerate}
    \item 
    We have $K\in\mathcal{K}$ and $\|\bT_{zw}(K)\|_{\Hinf}\leq \gamma$.
    \item There exists a positive definite matrix $P\succ0$ such that
    \begin{equation}\label{eq:BRL}
        \begin{bmatrix}    
        A_K^{\tr} P A_K-P+Q+K^\tr RK & A^{\tr}_K P B_w\\ B_w^{\tr} P A_K & B_w^{\tr} P B_w-\gamma^2 I\end{bmatrix} \preceq 0
    \end{equation}
where $A_K=A+BK$.
\end{enumerate}
\end{proposition}
\begin{proof}
Notice that
$(A_K,B_w)$ is always controllable, since $B_w$ is full row rank.
For $K\in\mathcal{K}$,
\cite[Theorem 2]{rantzer1996kalman} guarantees that
 $\|\bT_{zw}(K)\|_{\Hinf}\leq \gamma$ if and only if there exists a symmetric matrix $P$ satisfying \cref{eq:BRL}.
 Then,
 if (1) holds, the inequality \cref{eq:BRL}
 implies
\begin{equation*}
A_K^\tr P A_K-P+Q+K^\tr RK\preceq0
\end{equation*}
By the Schur stability of $A_K$ and $Q+K^\tr RK\succ 0$, we have $P\succ0$.
This implies (1)$\Rightarrow$(2).
Conversely, if (2) holds,
by $P\succ0$, the Schur stability of $A_K$ is also guaranteed by $A_K^\tr PA_K-P\preceq -Q-K^\tr RK\prec 0$,
which implies (2)$\Rightarrow$(1).
\end{proof}}

\Cref{proposition:BRL} allows us to rewrite the problem \cref{eq:policy-optimization-main} 
as a partial minimization 
of a \textit{hidden convex} function.
To see this, define the set 
\begin{equation*}
    \mathcal{L}_\mathrm{lft}=
    \left\{
    (K,\gamma,P)
    \mid
    P\succ 0,\,\gamma>0,\,
    \text{\cref{eq:BRL} holds}
    \right\}.
\end{equation*}
Then, problem \cref{eq:policy-optimization-main}
 can be rewritten into
\begin{equation}\label{eq:Hinf_opt_lif}
    \min_{K,\gamma,P}\quad \gamma \quad
    \text{subject to}\quad
    (K,\gamma,P)\in \mathcal{L}_\mathrm{lft}.
\end{equation}
To write \cref{eq:Hinf_opt_lif} into an unconstrained form, we define the lifted function $J_\mathrm{lft}$
\begin{equation}\label{eq:J_lift}
    J_\mathrm{lft}(K,\gamma,P)
    :=  
    \gamma + \delta_{\mathcal{L}_\mathrm{lft}}
    (K,\gamma,P),
\end{equation}
where $\delta_{\mathcal{L}_\mathrm{lft}}
    (K,\gamma,P)$ is the indicator function for the set $\mathcal{L}_\mathrm{lft}$,
    i.e., $\delta_{\mathcal{L}_\mathrm{lft}}
    (K,\gamma,P)=0$ for $ (K,\gamma,P)\in\mathcal{L}_\mathrm{lft}$ and 
     $\delta_{\mathcal{L}_\mathrm{lft}}
    (K,\gamma,P)=\infty$ if $ (K,\gamma,P)\notin\mathcal{L}_\mathrm{lft}$.
This allows us to represent the $\mathcal{H}_\infty$ function $J$ as
\begin{equation}\label{eq:Jk_parmin}
    J(K) = \min_{\gamma
    ,P}
    J_\mathrm{lft}(K,\gamma,P), \qquad \forall K \in \mathcal{K}.
\end{equation}
In other words, we can recover $J(\cdot)$ from $J_\mathrm{lft}(\cdot,\gamma,P)$ by eliminating $(\gamma,P)$ via the partial minimization.
The existence of such a partial minimizer for $K\in\mathcal{K}$ is guaranteed by \cref{lemma:basic-properties,proposition:BRL}.
Moreover, we can show the compactness of the sublevel set of $J_\mathrm{lft}(\cdot,\cdot,\cdot)$
{\cite[Proposition 3]{wang2026zeroth}}. 
\begin{lemma}\label{lemma:lifted_set-compactness}
    Suppose \cref{assumption:stablizability,assumption:K_nonempty} hold.
Assume $C=I_{n_x}$.
Then, 
for any $\nu>J^\star$,
the set $ \mathcal{L}_\mathrm{lft,\nu} = 
    \{
    (K,\gamma,P)
    \in \mathcal{L}_\mathrm{lft}
    \mid
    J_\mathrm{lft}(K,\gamma,P)
    \leq \nu
    \}$
    is compact.
\end{lemma}

Next, we show the hidden convex structure of this lifted function $J_\mathrm{lft}$.
In particular, 
the matrix inequalities \cref{eq:BRL} can be equivalently rewritten into a convex LMI via the Schur complement and a classical change of variables.
\begin{proposition}[Hidden convexity of $J_\mathrm{lft}$] \label{proposition:F_lft}
Suppose \cref{assumption:stablizability,assumption:K_nonempty} hold.
Assume $C=I_{n_x}$.
 Define the $C^\infty$ invertible mapping $\Upsilon:\mathbb{R}^{n_u\times n_x}\times(0,\infty)\times\mathbb{S}_{++}^{n_x}\to (0,\infty)\times\mathbb{R}^{n_u\times n_x}\times\mathbb{S}_{++}^{n_x}$:
\begin{equation*}
    \Upsilon(K,\gamma,P) = \left(\gamma,
    K(P/\gamma)^{-1},(P/\gamma)^{-1}
    \right).
\end{equation*}
Also, define 
the convex set
    $$\mathcal{F}_\mathrm{cvx}=
    \{
    (\gamma,Y,X)\mid
    X\succ 0,\,\gamma>0,\,
    \mathrm{LMI}(\gamma,Y,X)\succeq 0
    \},$$
where 
\begin{align}\label{eq:LMI-bounded-real-lemma}
    \mathrm{LMI}(\gamma,Y,X)= 
  \begin{bmatrix}
X & 0 & (A X+B Y)^\tr & X Q^{1 / 2} & Y^\tr R^{1 / 2} \\ 0 & \gamma I & B_w^\tr & 0 & 0 \\ A X+B Y & B_w & X & 0 & 0 \\ Q^{1 / 2} X & 0 & 0 & \gamma I & 0 \\ R^{1 / 2} Y & 0 & 0 & 0 & \gamma I
\end{bmatrix}.
\end{align}   
Then, we have $\Upsilon(\mathcal{L}_\mathrm{lft})
    = \mathcal{F}_\mathrm{cvx}.$
\end{proposition}
\begin{proof}
By Schur complement, the matrix inequality \cref{eq:BRL} is equivalent to
\begin{equation}\label{eq:BLR-five-block}
\begin{bmatrix}
    P & 0 & A_K^\tr P & Q^{1 / 2} & K^{\tr} R^{1 / 2} \\ 0 & \gamma^2 I & B_w^{\tr} P & 0 & 0 \\ P A_K & P B_w & P & 0 & 0 \\ Q^{1 / 2} & 0 & 0 & I & 0 \\ R^{1 / 2} K & 0 & 0 & 0 & I
\end{bmatrix}\succeq0.
\end{equation}
By pre- and post-multiplying this by 
$$\operatorname{diag}\left(\sqrt{\gamma} P^{-1}, \frac{1}{\sqrt{\gamma}} I,\sqrt{\gamma} P^{-1}, \sqrt{\gamma} I, \sqrt{\gamma} I\right)$$
equivalently yields 
$\mathrm{LMI}(\gamma,K(P/\gamma)^{-1},(P/\gamma)^{-1})\succeq0$
 with \cref{eq:LMI-bounded-real-lemma}.
 Then, it is now clear that  $\Upsilon(\mathcal{L}_\mathrm{lft})
    = \mathcal{F}_\mathrm{cvx}.$
\end{proof}

From the convex set $\mathcal{F}_{\mathrm{cvx}}$ in \cref{proposition:F_lft}, we define the following convex function:
\begin{equation} \label{eq:J_cvx}
    J_\mathrm{cvx}(\gamma,Y,X)
    := 
    \gamma
    + \delta_{\mathcal{F}_\mathrm{cvx}}(\gamma,Y,X). 
\end{equation}
From \cref{proposition:F_lft}, we immediately observe that 
\begin{align*}
    J_\mathrm{lft}(K,\gamma,P) &= 
    (J_\mathrm{cvx}\circ \Upsilon)
    (K,\gamma,P), \qquad \forall(K,\gamma,P)\in \mathcal{L}_\mathrm{lft} \\
J_\mathrm{cvx}(\gamma,Y,X) &= 
    (J_\mathrm{lft}\circ \Upsilon^{-1})
    (\gamma,Y,X) \qquad \forall (\gamma,Y,X)\in\mathcal{F}_\mathrm{cvx}.
\end{align*}
This implies that
$J_\mathrm{lft}$ admits an equivalent convex reformulation by the invertible mapping $\Upsilon$.

We will derive the weak PL inequality \cref{eq:weakPL} using the partial minimization \cref{eq:Jk_parmin} and the convex representation \cref{eq:J_cvx}.   
We first summarize a simple consequence of the hidden convexity of the lifted
formulation. For a convex function, the usual subgradient lower bound implies
a pointwise error bound in terms of the minimal-norm subgradient. Since
$J_{\rm lft}$ is obtained from the convex function $J_{\rm cvx}$ through the
smooth invertible change of variables $\Upsilon$, the same type of bound holds
for $J_{\rm lft}$, up to the local conditioning of the inverse map
$\Pi=\Upsilon^{-1}$.
\begin{lemma}\label{lemma:weak-PL-lft}
Suppose \cref{assumption:stablizability,assumption:K_nonempty} hold and
$C=I_{n_x}$.  
Let $
    z:=(K,\gamma,P)\in\mathcal L_{\rm lft}.
$ 
Assume $K$ is not optimal, i.e., $J(K)>J^\star$.
Then, we have 
\begin{equation} \label{eq:weak-PL-lft-lemma}
    J_{\rm lft}(z)-J^\star
    \le
    \beta_z\,\operatorname{dist}(0,\partial J_{\rm lft}(z)),
\end{equation}
where we may choose
$
    \beta_z
    :=
    \|D\Pi(\Upsilon(z))\|\,
    \|\Upsilon(z)-\bar z^\star\|,
    \,
    \Pi:=\Upsilon^{-1},
$ 
and $\bar z^\star\in\arg\min_{ z\in\mathcal F_{\rm cvx}}J_{\rm cvx}(z)$. Here $D\Pi(\Upsilon(z))$ denotes the Jacobian of $\Pi$ at $\Upsilon(z)$. 
\end{lemma}

\begin{proof}
Denote $\bar z=\Upsilon(z)$.
Since $J_{\rm cvx}$ is convex, for any
$\bar G\in\partial J_{\rm cvx}(\bar z)$ and any minimizer $\bar z^\star$,
\[
    J_{\rm cvx}(\bar z)-J^\star
    \le
    \langle \bar G,\bar z-\bar z^\star\rangle
    \le
    \|\bar G\|\,\|\bar z-\bar z^\star\|.
\]
Recall that
$
    J_{\rm lft}=J_{\rm cvx}\circ\Upsilon$, 
    $J_{\rm cvx}=J_{\rm lft}\circ\Pi$, and
    $\Pi=\Upsilon^{-1}$. 
By the chain rule for subdifferential \cite[Theorem 10.6]{rockafellar2009variational}, for any
$G\in\partial J_{\rm lft}(z)$, we have
\[
    \bar G:=D\Pi(\bar z)^\tr G
    \in
    \partial J_{\rm cvx}(\bar z).
\]
Therefore, 
\[
\begin{aligned}
    J_{\rm lft}(z)-J^\star
    =
    J_{\rm cvx}(\bar z)-J^\star  
    &\le
    \|\bar G\|\,\|\bar z-\bar z^\star\|  \\
    &\le
    \|D\Pi(\bar z)\|\,\|G\|\,\|\bar z-\bar z^\star\|.
\end{aligned}
\]
Taking the infimum over all $G\in\partial J_{\rm lft}(z)$ gives
\[
    J_{\rm lft}(z)-J^\star
    \le
    \|D\Pi(\bar z)\|\,\|\bar z-\bar z^\star\|\,
    \operatorname{dist}(0,\partial J_{\rm lft}(z)).
\]
This completes the proof.
\end{proof}
In \cref{subsection:proof-weakPL-lft}, we provide a more explicit computation for the Jacobian $D\Pi(\cdot)$. 
Next, from \cref{proposition:BRL}, for $K\in\mathcal{K}$ and $\hat \gamma=J(K)$,  
there exists a positive definite matrix $\hat P\succ 0$ such that
\begin{equation*}
     J(K) = J_\mathrm{lft}(K,\hat \gamma,\hat P),\qquad
     (K,\hat \gamma,\hat P)\in\mathcal{L}_\mathrm{lft}.
\end{equation*}
In other words, a partial minimizer
   $ (\hat \gamma,\hat P)\in \arg\min_{\gamma, P} J_\mathrm{lft}(K, \gamma, P)$ exists for any $K \in \mathcal{K}$.
We can evaluate the subdifferential of $J_\mathrm{lft}$ at such a partial minimizer using the subdifferential $\partial J(K)$.
\begin{lemma}\label{lemma:partial-minimization}
Suppose \cref{assumption:stablizability,assumption:K_nonempty} hold, and $C=I_{n_x}$.
Then, for $K\in\mathcal{K}$, we have
\begin{equation}\label{eq:partial-min-subgradient}
    \partial J(K)\times \{0\}
    \times \{0\}
    \subseteq 
    \bigcup_{(\gamma_K,P_K)\in  {\arg\min}_{\gamma,P} 
     J_{\mathrm{lft}}(K,\gamma,P)
     }
    \partial J_\mathrm{lft}(K,\gamma_K ,P_K).
\end{equation}
\end{lemma}
The proof uses the subdifferential formula for parametric minimization
\cite[Theorem 10.13]{rockafellar2009variational}. We provide further details in \cref{appendix:partial-minimization}. 

With \cref{lemma:weak-PL-lft,lemma:partial-minimization}, we are now ready to establish \cref{theorem:weak-PL}.

\vspace{1mm}
\noindent\textbf{Proof of \cref{theorem:weak-PL}.}
For any non-optimal gain $K\in\mathcal{K}_\nu$,
choose $G_K\in\arg\min_{G\in\partial J(K)}\|G\|_F.$
Then, by \cref{lemma:partial-minimization},
there exists
$(\gamma_K,P_K)\in
\arg\min_{\gamma,P}J_{\rm lft}(K,\gamma,P)$
such that $(G_K,0,0)\in
\partial J_{\rm lft}(K,\gamma_K,P_K)$. 
Further, $K\in\mathcal{K}_\nu$ implies
$z_{K}=(K,\gamma_K,P_K)\in \mathcal{L}_{\mathrm{lft},\nu}$.
Consequently,
for $K\in\mathcal{K}$ satisfying $J(K)>J^\star$,
it follows 
from \cref{lemma:lifted_set-compactness,lemma:weak-PL-lft} that
\begin{align*}
    J(K)-J^\star
= 
J_\mathrm{lft}(K,\gamma_K,P_K)-
J^\star
\overset{\cref{eq:weak-PL-lft-lemma}}{\leq}&
\beta_{z_K}
\times \mathrm{dist}(0,\partial J_\mathrm{lft}(K,\gamma_K,P_K)) \\
\overset{\cref{eq:partial-min-subgradient}}{\leq}
&
\beta_{z_K}\times
\mathrm{dist}(0,\partial J(K)) \\
\leq \;&
\left(\max_{z\in     \mathcal{L}_{\mathrm{lft},\nu}}\beta_{z}\right)
    \times
\mathrm{dist}(0,\partial J(K)) ,
\end{align*}
where 
$\|G_K\|_F=\mathrm{dist}(0,\partial J(K))$. 
By the continuity of $\beta_{z}$ with respect to $z$,
the compactness of $\mathcal{L}_{\mathrm{lft},\nu}$ yields
$0<\max_{z\in 
    \mathcal{L}_{\mathrm{lft},\nu}}\beta_{z}<\infty$.
Hence, by taking
\begin{align}\label{eq:mu_K}
    \mu=
    \left(\max_{z\in 
    \mathcal{L}_{\mathrm{lft},\nu}}\beta_{z}\right)^{-1}>0,
\end{align}
we arrive at
the weak PL inequality in \cref{eq:weakPL}.
When $K$ is optimal, \cref{eq:weakPL} is obvious.
\hfill$\square$
\section{A first-order proximal bundle method}\label{section:algorithm}

Building on the analysis in \cref{section:analysis-1,section:analysis-2}, we now introduce a new first-order algorithm to solve the $\mathcal{H}_\infty$ policy optimization problem \cref{eq:policy-optimization-main}. Since \cref{eq:policy-optimization-main} is in general nonconvex and nonsmooth, our goal is to find an approximated stationary point, given in \Cref{def:inexact-stationary}. In \cref{subsec:PPM}, we introduce a conceptual proximal point method. In \cref{subsec:PBM}, we develop the proximal bundle method that is guided by the proximal point method.

Throughout this section, we require an initial stabilizing gain $K_0\in\mathcal{K}$. Denote the initial sublevel set $\mathcal{K}_0 :=\{K\in\mathcal{K}\mid J(K)\leq J(K_0)\}$.
We make the following assumption
to specify the weak convexity and Lipschitz constants
based on the initial gain $K_0$.
\begin{assumption}\label{assumption:sublevel-set_weak-convexity}
Let $K_0$ be stabilizing
and satisfy $J(K_0)>J^\star$.
Define the enlarged sublevel set\footnote{
We use the threshold $2J(K_0)$ for concreteness. Our discussion applies
with any constant greater than $J(K_0)$.}
\begin{equation}\label{eq:sublevel_sets}
    \widetilde{\mathcal{K}}_0 :=\{K\in\mathcal{K}\mid J(K)\leq 2J(K_0)\}.
\end{equation}
Let $m>0$ be a constant such that
$J$ is $m$-weakly convex on any convex subset of $\widetilde{\mathcal{K}}_0$, and let $\ell>0$ be a constant such that
$J$ is $\ell$-Lipschitz on $\widetilde{\mathcal{K}}_0$.
\end{assumption}
The existence of the constants $m$ and $\ell$ in \cref{assumption:sublevel-set_weak-convexity} is guaranteed by the compactness of $\widetilde{\mathcal{K}}_0$ (see \Cref{lemma:basic-properties} and \Cref{theorem:local_wc}). We also define another useful constant 
\begin{equation} \label{eq:minimum-distance}
    \underline{d}:=
\min_{K\in\mathcal{K}_0}\mathrm{dist}(K,\partial\widetilde{\mathcal{K}}_0). 
\end{equation} 
In other words, $\underline d$ is the minimum distance from the initial sublevel set $\mathcal K_0$ to the boundary of the enlarged sublevel set $\widetilde{\mathcal K}_0$. This quantity is strictly positive because $\mathcal K_0$ and $\widetilde{\mathcal K}_0$ are compact (cf. \cref{lemma:basic-properties}), and
$\mathcal K_0$ is strictly contained in the interior of $\widetilde{\mathcal K}_0$.

\subsection{Proximal point method (PPM)}
\label{subsec:PPM}

To solve \cref{eq:policy-optimization-main},
we here present a conceptual proximal point method (PPM).
Starting with the initial gain $\mathcal{K}_0 \in \mathcal{K}$, 
the PPM generates iterates by solving a regularized subproblem 
\begin{equation} \label{eq:PPM-update}
    K_{t+1} = \argmin_{K \in \mathcal{K}} J(K) + \frac{\rho}{2}\|K - K_t\|^2_F, \quad t = 0,1,2, \ldots,
\end{equation}
where $\rho>0$ is a proximal parameter. 
We will next show that, for sufficiently large $\rho$, the subproblem in \cref{eq:PPM-update} has a unique minimizer and produces iterates that remain in the initial
sublevel set $\mathcal K_0$. In particular, it suffices to choose 
\begin{equation} \label{eq:lower-bound-on-rho}
    \rho > \max \left\{m,\frac{2(J(K_0)-J^\star)}{\underline{d}^2}\right\}. 
\end{equation}
\begin{lemma} \label{lemma:PPM-iterates}
    Suppose \cref{assumption:stablizability,assumption:K_nonempty,assumption:sublevel-set_weak-convexity} hold.
Let the proximal parameter $\rho$ satisfy \cref{eq:lower-bound-on-rho}. Then each PPM subproblem in
\cref{eq:PPM-update} has a unique minimizer. Moreover, the iterates satisfy
$
    K_t\in\mathcal \mathcal{K}_0,\, \forall t\ge 0.
$ 
\end{lemma}
\begin{proof}
    We prove the claim by induction. Suppose $K_t\in\mathcal K_0$. From \cref{lemma:basic-properties}, the function $K \mapsto J(K) + \frac{\rho}{2}\|K - K_t\|^2_F$ is continuous and coercive on $\mathcal{K}$. Thus, there must exist at least one minimizer within $\mathcal{K}$ for the PPM subproblem in \cref{eq:PPM-update}. Let $K_{t+1}$ be any minimizer. Since $K_t$ is feasible for the proximal
subproblem, we have 
    \begin{equation} \label{eq:one-step-improvement}
    J(K_{t+1}) + \frac{\rho}{2}\|K_{t+1} - K_t\|_F^2 \leq J(K_t). 
    \end{equation}
    Now it is clear that $J(K_{t+1}) \leq J(K_t) \leq J(K_0)$, and thus $K_{t+1} \in \mathcal{K}_0$. 

    We next show that $K_{t+1}$ must be unique if we have \cref{eq:lower-bound-on-rho}. From \cref{eq:one-step-improvement}, we know
    $$
    \|K_{t+1} - K_t\|_F \leq \sqrt{\frac{2(J(K_t) - J(K_{t+1}))}{\rho}} \leq \sqrt{\frac{2(J(K_0) - J^\star)}{\rho}} < \underline{d}.
    $$
    Thus every minimizer lies in the convex ball $B_{\underline d}(K_t)\subset
\widetilde{\mathcal K}_0$. On this ball, $J$ is $m$-weakly convex, and hence
$
    K\mapsto J(K)+\frac{\rho}{2}\|K-K_t\|_F^2
$ 
is $(\rho-m)$-strongly convex.
Therefore, the proximal subproblem has a unique minimizer.

The induction starts from $K_0\in\mathcal{K}_0$, and the argument above shows
that $K_{t+1}\in\mathcal K_0$ whenever $K_t\in\mathcal K_0$. Hence all PPM
iterates are well-defined, unique, and remain in $\mathcal K_0$.
\end{proof}

\Cref{lemma:PPM-iterates} confirms that the PPM in \Cref{alg:PPM} is well-defined. With a simple argument, we can derive the following non-asymptotic convergence rate for finding a stationary point.

\begin{algorithm}[t]
\caption{Proximal point method
}\label{alg:PPM}
\begin{algorithmic}
\Require $K_0 \in \mathcal{K}$ and
choose $\rho > 0$ as in \cref{eq:lower-bound-on-rho}.
\For{$t=0,1,2, \ldots$}
    \State Update $K_{t+1}$ by solving \Cref{eq:PPM-update}; 
\EndFor
\end{algorithmic}
\end{algorithm}

\begin{proposition}[Convergence of the PPM]\label{theorem:PPM}
Suppose \cref{assumption:stablizability,assumption:K_nonempty,assumption:sublevel-set_weak-convexity} hold.
Fix a target accuracy $\epsilon>0$. Let the proximal parameter $\rho$ satisfy \cref{eq:lower-bound-on-rho}. Then
the PPM in \Cref{alg:PPM} achieves
    \begin{equation} \label{eq:stationary-measure}
    \min_{t\in\{1,\ldots,T\}}
    \mathrm{dist}(0,\partial J(K_{t})) \leq \epsilon
    \end{equation}
    after at most   $T$  iterations with
    \begin{equation}\label{eq:PPM_rate}
        T
        =
         \left\lceil
        \frac{2\rho(J(K_0)-J^\star)}{\epsilon^2}
         \right\rceil.
    \end{equation}
\end{proposition}

\begin{proof}
A direct telescope sum of \Cref{eq:one-step-improvement} leads to 
\begin{equation}\label{eq:PPM_telescope_sum}
    \sum_{t=0}^{T-1}
    \| K_{t+1} - K_t \|_F^2
    \leq \frac{2}{\rho}
    \left(
    J(K_0) - J(K_{T})
    \right)
    \leq 
    \frac{2}{\rho}
    \left(
    J(K_0) - J^\star
    \right). 
\end{equation}
As argued in the proof of \Cref{lemma:PPM-iterates}, each PPM subproblem in \Cref{eq:PPM-update} admits a unique minimizer within the ball $B_{\underline{d}}(K_t)\subset
\widetilde{\mathcal K}_0$, and the function $
    K\mapsto J(K)+\frac{\rho}{2}\|K-K_t\|_F^2
$ 
is $(\rho-m)$-strongly convex over this ball. So the optimality condition of \Cref{eq:PPM-update} gives 
\begin{equation} \label{eq:subgradient-PPM-subproblem}
    0 \in \partial J(K_{t+1}) + \rho(K_{t+1} - K_t) \quad \Rightarrow \quad -\rho(K_{t+1} - K_t) \in \partial J(K_{t+1}), 
\end{equation}
implying that $\rho \|K_{t+1}-K_t\|_F 
    \geq \mathrm{dist}(0,\partial J(K_{t+1}))$. Combining this with \cref{eq:PPM_telescope_sum} leads to 
    $$
    \begin{aligned}
    T \times \min_{t=\{0,1, \ldots, T-1\}} \frac{1}{\rho^2}\mathrm{dist}(0,\partial J(K_{t+1}))^2 &\leq 
    \frac{2}{\rho}
    \left(
    J(K_0) - J^\star
    \right) \\  \quad \Rightarrow \quad \min_{t=\{0,1, \ldots, T-1\}} \mathrm{dist}(0,\partial J(K_{t+1}))^2 &\leq \frac{2\rho}{T}
    \left(
    J(K_0) - J^\star
    \right).  
    \end{aligned}
    $$
    We thus arrive at the desired bound \cref{eq:PPM_rate} to ensure \cref{eq:stationary-measure}. 
\end{proof}

\cref{theorem:PPM} ensures that
we can find a feedback gain $K_t$ satisfying $\mathrm{dist}(0,\partial J(K_t)) \leq \epsilon$ for the $\mathcal H_\infty$ policy optimization problem \cref{eq:policy-optimization-main} within at most $\mathcal O(1/\epsilon^2)$ iterations.  
We have used two different sublevel sets
$\mathcal{K}_0$ and 
$\widetilde{\mathcal{K}}_0$ in the proof of \cref{theorem:PPM}.
This is to guarantee the existence of a convex subset of
$\widetilde{\mathcal{K}}_0$ containing both
$K_t$ and $K_{t+1}$, which enables the use of weak convexity in \cref{theorem:local_wc}.

The PPM in \Cref{alg:PPM} is not directly implementable since the regularized subproblem \cref{eq:PPM-update} is difficult to solve exactly. 
Nevertheless, the PPM is conceptually useful, and it motivates~our~design of a \textit{proximal bundle method}, which solves \cref{eq:PPM-update} approximately through computationally simple~iterations.

\subsection{Proximal bundle method (PBM) as an inexact PPM}
\label{subsec:PBM}

We here develop 
a \textit{proximal bundle method} (PBM) to solve the $\mathcal{H}_\infty$ policy optimization \cref{eq:policy-optimization-main}, which  can be viewed as an inexact PPM.

\subsubsection{A proximal bundle method for $\mathcal{H}_\infty$ optimization}

We now introduce the PBM, an iterative algorithm with a subroutine that
approximates the update of the exact PPM (\Cref{alg:PPM}) and returns an iterate $K_{t+1}$ satisfying a similar, quantifiable descent property to that of the PPM. We call the subroutine \texttt{ProxDescent}, listed in
\Cref{alg:proxdescent}. Motivated by \cite{liao2025proximal},
\texttt{ProxDescent} is designed as a bundle scheme tailored to the
$\mathcal H_\infty$ policy optimization problem. The overall PBM for solving
\cref{eq:policy-optimization-main} is listed in \Cref{alg:PBM_outerloop}. It
can be viewed as a double-loop algorithm: the outer loop follows the structure
of the PPM, while the inner loop, implemented by \texttt{ProxDescent}, ensures
a relaxed version of the descent condition \cref{eq:one-step-improvement} at each outer iteration.

One key step in \texttt{ProxDescent} is to construct a suitable lower approximation function $J_k(\cdot)$ for the locally convex function
$J(\cdot)+\frac{m}{2}\|\cdot-K_t\|_F^2$, and then to compute
\begin{equation}\label{eq:L_k+1}
    L_{k+1}
    =
    \operatorname*{argmin}_{L \in \mathbb{R}^{n_u \times n_y}} 
    \;
    J_{k}(L) + 
    \frac{\rho}{2}\|L-K_t\|_F^2. 
\end{equation}
Here, $k$ is another iteration count for the inner loop
and $\rho>0$ is a sufficiently large positive number. Note that the subproblem \cref{eq:L_k+1} is unconstrained and has no explicit stabilizing constraint $L \in \mathcal{K}$. As we will show in \Cref{lemma:feasibility_nullstep} later, each minimizer naturally satisfies $L_{k+1} \in \mathcal{K}$ when $\rho>0$ is sufficiently large (which corresponds to a sufficiently small stepsize). 
By \Cref{assumption:sublevel-set_weak-convexity}, the function $J(\cdot)+\frac{m}{2}\|\cdot-K_t\|_F^2$ is convex over any convex subset of $\widetilde{\mathcal{K}}_0$. 
 As we detail in \cref{assumption:conditions_Jk} below, this convex property allows us to construct $J_k(\cdot)$ as a piecewise linear function, satisfying the lower-approximation property:
\begin{equation}\label{eq:Jk-lower-approximation}
     J_{k}(L) 
     \leq J(L) + 
     \frac{m}{2}\|L-K_t\|_F^2,
     \quad
     \forall L\in \mathcal{U}_{K_t},
\end{equation}
where $\mathcal{U}_{K_t}$ is the largest open ball around $K_t$ on $\operatorname{int}(\widetilde{\mathcal{K}}_0)$.
We remark that \cref{eq:L_k+1} can be efficiently updated for the piecewise linear function $J_{k}(\cdot)$.

{\begin{algorithm}[t]
\caption{\texttt{ProxDescent}$(K_t,\beta,\rho)$}\label{alg:proxdescent}
\begin{algorithmic}[1]
\Require $K_t \in {\mathcal{K}}_0,\,\beta \in (0,1)$, and
$\rho >0$.
\For{$k$=0,1,2\ldots}
\State Construct $J_{k}$ as 
in \cref{assumption:conditions_Jk}; 
     \State Compute $L_{k+1}$ via \cref{eq:L_k+1}; 
     \Comment{For $J_k$ in \cref{eq:two-cut-model}, use the analytic solution \cref{eq:L_analytic}.} 
\If{\cref{eq:approx} is satisfied}
\State \Return $L_{k+1}$. 
\EndIf
\EndFor
\end{algorithmic}
\end{algorithm}
\begin{algorithm}[t]
\caption{Proximal bundle method
for  \cref{eq:policy-optimization-main}
}\label{alg:PBM_outerloop}
\begin{algorithmic}[1]
\Require $K_0 \in \mathcal{K}_,\,\beta \in (0,1),$ and
$\rho >0$.
\For{$t=0,1,2, \ldots$}
    \State $K_{t+1} = \texttt{ProxDescent}(K_t,\beta,\rho) 
    $ with \Cref{alg:proxdescent}.
\EndFor
\end{algorithmic}
\end{algorithm}}

Then, 
to determine the next point $K_{t+1}$,
we evaluate the quality of
the candidate point $L_{k+1}$ 
through
the following criterion:
\begin{equation}\label{eq:approx}
\beta\left(
    J(K_t) - J_{k}(L_{k+1})\right) \leq  J(K_t)-\left(J(L_{k+1})
    + \frac{m}{2}\|L_{k+1}-K_t\|_F^2
    \right)
\end{equation}
with a constant 
$\beta\in(0,1)$. If \cref{eq:approx} holds, we set $K_{t+1} = L_{k+1}$; otherwise, we refine the local model $J_{k+1}$ via \cref{assumption:conditions_Jk} and repeat the procedure above. This inner search procedure is listed in \Cref{alg:proxdescent}. This stopping criterion \cref{eq:approx} follows from standard proximal bundle methods \cite{diaz2023optimal,liao2025proximal},
and guarantees a descent property (see \cref{proposition:PBM-descent} below).

As for the (local) lower approximation $J_k$, we consider the following three types of convex functions.
These choices are standard in proximal bundle methods 
\cite{diaz2023optimal,liao2025proximal}.
\begin{assumption}\label{assumption:conditions_Jk}
At $k=0$,
{
let 
$L_0=K_t$ and
$J_0(\cdot) = 
J(L_0)+\langle G_0,\cdot -L_0\rangle$ with $G_0\in \partial J(L_0)$.}
At $k$-th iteration for $k\geq 1$, the lower approximation $J_k(\cdot)$ is given by one of the following functions:
\begin{subequations}
\begin{enumerate}
    \item 
    Full-memory model:
    \begin{equation} \label{eq:full-memory}
    \begin{aligned}
    J_{k}(L)
    =\max_{i=0,\ldots k}
    J(L_i)+ \frac{m}{2}\|L_i-K_t\|^2_F
    + \langle 
    G_{i},L-L_{i}
    \rangle;
    \end{aligned}
    \end{equation}
    \item 
    Two-cut model: 
    \begin{equation} \label{eq:two-cut-model}
    J_{k}(L)
    =\max
    \left\{
    J(L_{k})
    + \frac{m}{2}\|L_{k}-K_t\|^2
    + \langle G_k,L-L_k\rangle,
    J_{k-1}(L_k)
    + \langle S_k,L-L_k \rangle
    \right\};
    \end{equation}
    {\item A more general form:
    \begin{equation} \label{eq:model-generel}
    J_{k}(L)
    \!=\!\max
    \left\{\mathcal{M}_k(L),
    J(L_{k})
    + \frac{m}{2}\|L_{k}-K_t\|^2
    + \langle G_k,L-L_k\rangle,
    J_{k-1}(L_k)
    \!+\! \langle S_k,L-L_k \rangle
    \right\},
    \end{equation}}
\end{enumerate}
\end{subequations}
where 
$G_i -m(L_i-K_t) \in \partial J(L_i)$ and $S_k = \rho(K_t-L_k) \in \partial J_{k-1}(L_k)$.
In the model 3), $\mathcal{M}_k(\cdot)$ is a possibly nonsmooth function that is 
$\ell_\mathcal{M}$-Lipschitz on $\widetilde{\mathcal{K}}_0$,
convex on $\mathbb{R}^{n_u\times n_y}$, and satisfies $J(L)+\frac{m}{2}\|L-K_t\|_F^2\geq\mathcal{M}_k(L)$ for all $L\in U_{K_t}$,
 where $U_{K_t}$ is the largest open ball around $K_t$ on $\widetilde{\mathcal{K}}_0.$
\end{assumption}

The first full-memory model uses all the past iterates to approximate $J(\cdot)+\frac{m}{2}\|\cdot-K\|_F^2$.
The second two-cut model consists of two affine functions, i.e., the first-order approximation of $J(\cdot)+\frac{m}{2}\|\cdot-K\|_F^2$ at $L_k$ and that of $J_{k-1}$ at $L_k$, which represents the aggregation of the previous iterates. 
The subproblem \cref{eq:L_k+1} using $J_k(\cdot)$ in \Cref{assumption:conditions_Jk} allows much more efficient computation than the original PPM in \Cref{alg:PPM}, since it is reduced to a quadratic program (QP).
Furthermore, when using the two-cut model \cref{eq:two-cut-model}, the subproblem \cref{eq:L_k+1} admits a closed-form analytical solution: 
        \begin{align}\label{eq:L_analytic}
        {L}_{k+1}\!=\!
        K_t
        - \frac{1}{\rho}\left(
        \theta_{k}G_k + (1-\theta_k)S_k
        \right), \quad
        \theta_{k} = \min
        \left\{
        1, \frac{\rho(J(L_k)
        +\frac{m}{2}\|L_k-K_t\|_F^2
        -J_{k-1}(L_k))}{\|G_k-S_k\|_F^2}
        \right\},
    \end{align}
where we set $\theta_k=1$ when $G_k=S_k$.
The derivation of this closed-form solution is simple; see e.g., \cite[Claim 1]{diaz2023optimal}, and \cite[Appendix B.6]{liao2025proximal}. 
The third model adds flexibility to the two-cut model by allowing another lower approximation $\mathcal{M}_k(\cdot)$.
In particular, by appropriately constructing $\mathcal{M}_k(\cdot)$,
one can incorporate the existing spectral bundle model by \cite{apkarian2006nonsmooth,apkarian2009proximity} into our framework; see \Cref{remark:SBM}.

\begin{lemma}[Implication of \cref{assumption:conditions_Jk}]
\label{corollary:lower-bound}
      Fix $K_t\in\mathcal{K}_0$. Let $\{L_k\}$ be the sequence generated by  \Cref{alg:proxdescent}. Assume 
      that
      $\|L_k-K_t\|_F<\underline{d} .$ For each $k \geq0$, the approximate function $J_k$ in \Cref{alg:proxdescent} satisfies \cref{eq:Jk-lower-approximation}, i.e., 
     $
         J_{k}(L) 
     \leq J(L) + 
     \frac{m}{2}\|L-K_t\|_F^2,\; \forall L\in \mathcal{U}_{K_t}.
     $
\end{lemma}
\begin{proof}
     The assumption that $\{L_k\}\subset \mathbb{B}_{ \underline{d} }(K_t) $ guarantees $\{L_k\} \subset \widetilde{\mathcal{K}}_0$. From \cref{assumption:sublevel-set_weak-convexity}, the function $J$ is $m$-weakly convex on $\mathcal{U}_{K_t} $. For all $i \geq 0$, it holds that 
\begin{align}
    \label{eq:lower}
    J(L_i)+ \frac{m}{2}\|L_i-K_t\|^2_F
    + \langle 
    G_{i},L-L_{i}
    \rangle \leq  J(L)+ \frac{m}{2}\|L-K_t\|^2_F, \quad \forall L \in  \mathcal{U}_{K_t}, 
\end{align}
where $G_i - m(L_i - K_t)\in \partial J(L_i)$.
We can show \cref{eq:Jk-lower-approximation} for each model as follows.
\begin{itemize}
    \item Full-memory model \cref{eq:full-memory}: From \cref{eq:lower}, it is clear that the full-memory model \cref{eq:full-memory} satisfies \cref{eq:Jk-lower-approximation}.
    \item Two cut model \cref{eq:two-cut-model}: \cref{assumption:conditions_Jk} assumes that 
$
    J_0(\cdot) = 
J(K_t)+\langle G_0,\cdot -K_t\rangle = J(K_t) + \frac{m}{2}\|K_t - K_t\|^2+\langle G_0,\cdot -K_t\rangle
$ 
with $G_0 = G_0  - m(K_t - K_t)\in \partial J(K_t)$. Thus, $J_0$ satisfies $J_{0}(K) 
     \leq J(L) + 
     \frac{m}{2}\|L-K_t\|_F^2, \forall L\in \mathcal{U}_{K_t}$.
     Next, since $S_1 = \rho (K_t-L_1) \in \partial J_0(L_1)$ and $J_0$ is a convex function, we have 
\begin{align*}
    J_{0}(L_1)
    + \langle S_1,L-L_1 \rangle \leq J_0 (L ) \leq J(L) + 
     \frac{m}{2}\|L-K_t\|_F^2, \quad \forall L \in  \mathcal{U}_{K_t}. 
\end{align*}
Thus, $J_1(L) \leq J(L) + 
     \frac{m}{2}\|L-K_t\|_F^2$. Inductively, we have $J_k(L) \leq J(L) + 
     \frac{m}{2}\|L-K_t\|_F^2,$ for all $k.$
    \item A more general model \cref{eq:model-generel}: The argument follows the same as that for the two-cut model since $J(L)+\frac{m}{2}\|L-K_t\|_F^2\geq\mathcal{M}_k(L)$ for all $L\in \mathcal{U}_{K_t}$.
\end{itemize}
Hence, the three models in \cref{assumption:conditions_Jk} satisfy the lower approximation property \cref{eq:Jk-lower-approximation}.
\end{proof}

If the acceptance test \cref{eq:approx} is satisfied, we can guarantee a decrease in function value. 
This also confirms $L_{k+1}\in\mathcal{K}_0$ when \cref{eq:approx} holds.
Note that applying to this inequality the same proof strategy as \cref{theorem:PPM} does not provide a complexity result by itself, since $\rho(K_t-L_{k+1})\in \partial J_k(L_{k+1}) \neq \partial J(L_{k+1})$ in general.
\begin{proposition}\label{proposition:PBM-descent}
     Suppose \Cref{assumption:stablizability,assumption:K_nonempty,assumption:sublevel-set_weak-convexity,assumption:conditions_Jk} hold.
     Fix $K_t\in\mathcal{K}_0$. Let $\{L_k\}$ be the sequence generated by  \Cref{alg:proxdescent}. Assume $\{L_k\}\subset
     \mathbb{B}_{ \underline{d} }(K_t)
     $.
     Then, if the acceptance test \cref{eq:approx} holds, we have\begin{equation}\label{eq:descent:PBM}
         J(L_{k+1})\leq J(K_t)
         - \frac{m}{2}\|L_{k+1}-K_t\|_F^2.
     \end{equation}
\end{proposition}
\begin{proof}
To show \cref{eq:descent:PBM}, we show that $J(K_t)\geq J_k(L_{k+1})$. Then dropping the nonnegative term $\beta(J(K_t) - J_k(L_{k+1}))$ with a simple rearrangement in \cref{eq:approx} yields the result.

 From \cref{corollary:lower-bound}, we have that 
$J(K_t)\geq J_k(K_t)$.
On the other hand, we have 
\begin{equation*}
    J_k(K_t) =
    J_k(K_t)+\frac{\rho}{2}\|K_t-K_t\|_F^2
    \geq J_k(L_{k+1}) +
    \frac{\rho}{2}\|L_{k+1}-K_t\|_F^2
    \geq J_k(L_{k+1})
\end{equation*}
where the first inequality is by \cref{eq:L_k+1}.
Combining the above two inequalities gets $J(K_t)\geq J_k(L_{k+1})$.
\end{proof}

In this proposition, we assume $\{L_k\}\subset
     \mathbb{B}_{ \underline{d} }(K_t)
     $ to properly define the model $J_k$ and its lower-approximation property.
We will formally ensure this assumption in \cref{lemma:feasibility_nullstep} below.

\begin{remark}[A spectral bundle model]\label{remark:SBM}
In designing a bundle model satisfying \cref{assumption:conditions_Jk},
one promising choice of function $\mathcal{M}_k(\cdot)$ in the third model is the following \textit{spectral model}:
\begin{equation*}
    \mathcal{M}_k(L)=
    \max_{\theta=(\omega,u,v)\in\Theta_k}
    \{
    \varphi\left(\theta, K_t\right)+\left\langle\nabla_K \varphi\left(\theta, K_t\right), L-K_t\right\rangle\},
\end{equation*}
where 
$\Theta_k$ is a (possibly finite) compact subset of the set $\Theta$, and
$\varphi(\cdot,\cdot)$ is defined in \cref{eq:local-smooth-hinf}.
For any sufficiently large $m>0$,
the lower approximation property \cref{eq:Jk-lower-approximation} holds \cite[Lemma 2.1]{apkarian2009proximity}.
In this model, we take the first approximation for the smooth function $\varphi(\theta,K)$ with respect to $K$, not for the cost $J$ itself.
The continuous-time version of this model was first proposed in \cite{apkarian2006nonsmooth,apkarian2009proximity}.
In contrast with the full-memory and two-cut models in \cref{assumption:conditions_Jk}, this model can capture the eigenvalue and frequency structure of the cost, which may particularly enhance the convergence performance near stationary points. 
This model also admits efficient implementation when $\Theta_k$ is a finite set.
In particular, the problem 
\cref{eq:L_k+1} is reduced to a simple QP, and its dual problem can be solved very efficiently \cite[Section 4]{apkarian2009proximity}.
\hfill$\square$
\end{remark}

\subsection{Non-asymptotic convergence guarantees}\label{section:proof-main_theorem}

The proximal bundle method in \Cref{alg:PBM_outerloop} is implementable, since each iteration involves only simple computations. We now present our main convergence result for this algorithm.
{Here, 
we define the $(\eta,\epsilon)$-stationarity (\Cref{def:inexact-stationary}) 
for the open domain $V=\operatorname{int}(\widetilde{\mathcal{K}}_0)$,
and consider
the restriction $J|_{\operatorname{int}(\widetilde{\mathcal{K}}_0)}:\operatorname{int}(\widetilde{\mathcal{K}}_0)\to\mathbb{R}$ of $J:\mathcal{K}\to\mathbb{R}$.
This restriction is sufficient for our convergence analysis, since we will show that all inner and outer iterates remain in the interior of $\widetilde{\mathcal{K}}_0$ (\cref{lemma:feasibility_nullstep}).
For brevity,
$J|_{\operatorname{int}(\widetilde{\mathcal{K}}_0)}(L)$ 
for $L\in\operatorname{int}(\widetilde{\mathcal{K}}_0)$
is simply denoted by $J(L)$.}

We first show that, when $\rho$ is sufficiently large, all inner-loop iterates generated by \Cref{alg:proxdescent} remain stabilizing. This is essential for ensuring feasibility throughout the inner loop. Indeed, by the construction of the model $J_k(\cdot)$ in \cref{assumption:conditions_Jk}, we can control the distance between $L_{k+1}$ in \cref{eq:L_k+1} and the proximal center $K_t$. This allows us to guarantee that $L_{k+1}$ remains in the interior of the enlarged feasible sublevel set $\widetilde{\mathcal K}_0$. We have the following result. 

\begin{lemma}[Feasibility of inner iterates]\label{lemma:feasibility_nullstep}
 Suppose \Cref{assumption:stablizability,assumption:K_nonempty,assumption:sublevel-set_weak-convexity,assumption:conditions_Jk} hold.
Let $K_t\in \mathcal{K}_0$
and $\rho >\max\{\frac{\ell}{\underline{d}}+m,\frac{\ell_\mathcal{M}}{\underline{d}}\}$, where $m,\ell$ are given in \cref{assumption:sublevel-set_weak-convexity}, $\underline{d}$ is defined in \cref{eq:minimum-distance},
and $\ell_\mathcal{M}\geq0$ is defined in \cref{assumption:conditions_Jk}.
For the lower approximation $J_k(\cdot)$ from \Cref{assumption:conditions_Jk}, the minimizer $L_{k+1}$ in \cref{eq:L_k+1} satisfies $
    \|L_{k+1}-K_t\|<\underline{d}$ 
     and thus  $ L_{k+1}\in \operatorname{int}(\widetilde{\mathcal{K}}_0)$
    for all $
     k\geq 0$. 
\end{lemma}

The detailed proof is given in \cref{appendix:proof-feasibility-nullstep}. 
It is worth noting that we can ensure the feasibility of all points $\{L_{k}\}$, generated by the unconstrained minimization \cref{eq:L_k+1} with $J_k$, \textit{without explicitly imposing the nonconvex constraint $L\in\mathcal{K}$.}
This property is not trivial in nonsmooth optimization and is not guaranteed in standard subgradient methods \cite{watanabe2025policy}. 

In the following, we establish the total iteration complexity of the PBM (\Cref{alg:PBM_outerloop}) to find an $(\eta,\epsilon)$-stationary point. Using the double-loop structure, we separately analyze the iteration complexity for the inner loop (\Cref{alg:proxdescent}) and outer loop (\Cref{alg:PBM_outerloop}). The total number of iterations then follows as their product. To establish the complexity results, we prepare the following lemma, an adaptation of \cite[Lemma A.4]{liao2025proximal} for the $\mathcal{H}_\infty$ optimization. This shows that a small gap from the exact PPM leads to a small inexact subgradient guarantee.
{\begin{lemma}\label{lemma:inner-loop-2}
    Suppose \Cref{assumption:stablizability,assumption:K_nonempty,assumption:sublevel-set_weak-convexity,assumption:conditions_Jk} hold.
    Let $K\in\mathcal{K}_0$ and $\rho>\max\{\frac{2\ell}{\underline{d}}-m,0\}$.
    Denote $\hat{K} \! =\! \arg\min_{L} J(L) + \frac{m+\rho}{2}\|L-K\|_F^2$, and $\Delta_K \!=\! J(K) - \min_{L} \{J(L) + \frac{m+\rho}{2}\|L-K\|_F^2\}$. Then we have 
    \begin{align*}
    U=\rho (K-\hat K) \in  \partial_{\Delta_K} J(K).
    \end{align*}
    Moreover, the inexact subgradient $U$ is bounded as $\|U\|_F\leq \sqrt{2\rho \Delta_K}$.
\end{lemma}}

The proof is presented in \cref{appendix:proof-inner-lemma-2}. Now, we give the following lemma for the inner-loop by extending \cite[Lemma 5.2]{diaz2023optimal}. Combining this lemma with \cref{lemma:inner-loop-2}, we can evaluate the iteration complexity of the inner loop in \Cref{alg:proxdescent}.

\begin{lemma}[Inner-loop complexity]
\label{lemma:innerloop}
 Suppose \Cref{assumption:stablizability,assumption:K_nonempty,assumption:sublevel-set_weak-convexity,assumption:conditions_Jk} hold.
Let $K_t \in \mathcal{K}_0$ and {$\rho >
\max\{ \frac{\ell}{\underline{d}}+m,\frac{\ell_\mathcal{M}}{\underline{d}} ,\frac{2\ell}{\underline{d}}-m\} $}. 
Let $\Delta_{t}$ be the proximal gap defined by
    $$\Delta_{t}
    = J(K_{t})-J_{1/(m+\rho)}(K_t), \quad \text{with} \quad  
    J_{1/(m+\rho)}(K_t):=\min_L
    \left\{
    J(L) + \frac{m+\rho}{2}
    \|L-K_{t}\|_F^2
    \right\}.$$
Then if $\Delta_t\neq0$,
\Cref{alg:proxdescent} 
terminates
in at most $T_{\rm inner}$ iterations with 
\begin{equation} \label{eq:inner-iterations}
    T_{\rm inner}  
    \leq 
    \frac{
    8(\ell+m\underline{d})^2
    }{
    (1-\beta)^2 \rho \Delta_{t}
    }+1,
\end{equation}
\end{lemma}
\begin{proof}
            Recall that \cref{lemma:feasibility_nullstep}
    ensures $\|L_k-K_t\|_F < \underline{d}$,
    implying that $\{L_k\}$ is contained in the ball $\mathbb{B}_{ \underline{d}}(K_t)$.
    Further, by $\rho> \frac{2\ell}{\underline{d}}-m,$
    the true proximal point
    $\hat K_t\in\argmin_{L}J(L)+\frac{m+\rho}{2}\lVert L-K_t\rVert_F^2$
  satisfies
\begin{equation*}
\frac{m+\rho}{2}\|\hat{K}_t-K_t\|_F^2 \leq J(K_t)-J(\hat{K}_t) \leq \ell \|\hat{K}_t-K_t\|_F,
\end{equation*}
which yields
\begin{equation*}
    \|\hat{K}_t-K_t\|_F \leq \frac{2\ell}{m+\rho}<\underline{d}.
\end{equation*}
Thus the true proximal point and all trial points lie in the common convex ball $\mathbb{B}_{\underline{d}}(K_t)$ on $\widetilde{\mathcal{K}}_0$.
Then, all convexity and minorization inequalities used in the proof of  \cite[Lemma 5.2]{diaz2023optimal} remain valid for the unconstrained problem $\min_{K }\; J(K)+ \frac{m}{2}\|K - K_{t}\|_F^2$
    with the fixed center point $K_{t}$.
Hence the same null-step proof applies locally. 
    In particular, from \cite[Lemma 5.2]{diaz2023optimal}, we know that the number of null steps before \cref{eq:approx} holds  is bounded by $$ T_{\rm inner}\leq \frac{8 G^2}{(1 - \beta)^2 \rho \Delta_{t}},$$
    where $G=\max_{i\leq T}\|G_i\|_F$.
    Finally, \cref{eq:inner-iterations} follows from 
    the Lipschitz continuity of $J(\cdot)$
    and \cref{lemma:feasibility_nullstep},
    as 
    $G\leq \ell +m\sup_{i\leq T}\|L_i-K_t\|_F\leq \ell+m\underline{d}$.
    \end{proof}

Now, combining \cref{lemma:inner-loop-2,lemma:innerloop}, it is not difficult to see that {$\mathcal{O}(1/\Delta_t)=\mathcal{O}\left(
    \max\left\{
    \frac{1}{\eta^2},\frac{1}{\epsilon}
    \right\}
    \right)$ iterations} are required for each inner-loop (\Cref{alg:proxdescent}),
    unless $K_t$ is an
$\left(\eta, \epsilon
    \right)$-stationary point.
Indeed, \Cref{lemma:inner-loop-2,lemma:innerloop} ensure that:
\begin{enumerate}
    \item if 
$\Delta_{t} \leq \sigma:=
\min\left\{
{\eta^2}/{2\rho},\epsilon
\right\}$, then $K_{t}$
    is an
$\left(\eta, \epsilon
    \right)$-stationary point;
    \item if $K_t$ is not 
an
$\left(\eta, \epsilon
    \right)$-stationary point, we must have $\Delta_t > \sigma$, and thus the inner loop takes 
at most $\mathcal{O}(1/\Delta_t)=\mathcal{O}\left(
    \max\left\{
    \frac{1}{\eta^2},\frac{1}{\epsilon}
    \right\}
    \right)$ iterations.
\end{enumerate}

Next, we present the iteration complexity of the outer loop (\Cref{alg:PBM_outerloop}).
Thanks to the descent condition \cref{eq:approx},
we can largely perform a similar analysis strategy based on a telescoping sum to the PPM in \cref{theorem:PPM}.
Similar to the inner loop, we prove that the best gap $\Delta_t$ will eventually become small with $O(1/T)$, leading to
the complexity below.
Its proof is deferred to \cref{subsubsection:proof_outerloop}.
{\begin{lemma}[Outer-loop complexity]\label{lemma:outerloop}
Suppose \Cref{assumption:stablizability,assumption:K_nonempty,assumption:sublevel-set_weak-convexity,assumption:conditions_Jk} hold. Choose $\rho$ as in~\cref{lemma:innerloop}.
Then, 
\Cref{alg:PBM_outerloop} with $K_0 \in \mathcal{K}$ and $\beta\in(0,1)$ 
guarantees that $\{K_t\}\subset \mathcal{K}_0$ and
finds an $(\eta,\epsilon)$-stationary point
after
at most 
\begin{equation*}
    T_{\rm outer}
    =
    \left\lceil
\left(\frac{J(K_0)-J^\star}{\beta}\right)
\max\left\{\frac{2\rho}{\eta^2},\frac{1}{\epsilon}\right\}
\right\rceil
=
    \mathcal{O}
    \left(
    \max\left\{
    \frac{1}{\eta^2},
    \frac{1}{\epsilon}\right\}
    \right)
\end{equation*}
outer iterations.
\end{lemma}}

Combining these two lemmas, we obtain the following convergence guarantee.

\begin{theorem}[Non-asymptotic convergence]\label{theorem:main-result}
    Consider the $\mathcal{H}_\infty$ policy optimization \cref{eq:policy-optimization-main}. 
    Suppose \Cref{assumption:stablizability,assumption:K_nonempty,assumption:sublevel-set_weak-convexity,assumption:conditions_Jk} hold.
    Choose $\rho>0$ as in \Cref{lemma:innerloop}, and let $\eta>0,\epsilon > 0$. 
    For the proximal bundle method in \Cref{alg:PBM_outerloop} with the inner-loop \Cref{alg:proxdescent}, we have
    $\{K_t\}\subset\mathcal{K}_0$ in \Cref{alg:PBM_outerloop}
    and 
    $\{L_k\}\subset \operatorname{int}(\widetilde{\mathcal{K}}_0)$ in
    \Cref{alg:proxdescent},
    and we obtain an $(\eta,\epsilon)$-stationary point in at most 
    \begin{equation}\label{eq:order_main_theorem}
   { \bigO 
    \left(
    \max\left\{\frac{1}{\eta^4},\frac{1}{\epsilon^2}\right\}
    \right )}
    \end{equation}
    total iterations.
\end{theorem}

\begin{proof}
    As argued above, before reaching an $(\eta,\epsilon)$-stationary point, each call of \Cref{alg:proxdescent} takes at most {$\mathcal{O}(1/\Delta_t)=\mathcal{O}\left(
    \max\left\{
    \frac{1}{\eta^2},\frac{1}{\epsilon}
    \right\}
    \right)$ iterations}. From \cref{lemma:outerloop}, the outer-loop takes at most $\mathcal{O}\left(
    \max\left\{
    \frac{1}{\eta^2},\frac{1}{\epsilon}
    \right\}
    \right)$ iterations. Multiplying them together leads to the desired complexity \cref{eq:order_main_theorem} for the total number of iterations. 
\end{proof}

 To our best knowledge, this is the first deterministic, non-asymptotic convergence rate for $\Hinf$ policy optimization with an anytime feasibility guarantee.
By setting $\eta=\Theta(\epsilon)$,
\cref{theorem:main-result} implies $\mathcal{O}(\epsilon^{-4})$ iterations,
which matches the standard deterministic results for unconstrained weakly convex functions \cite{davis2019stochastic,liang2023proximal,liao2025proximal,kong2024cost,li2026optimal}.
\cref{theorem:main-result} can be viewed as an extension of the unconstrained and globally weakly convex counterpart \cite{liao2025proximal} to $\mathcal{H}_\infty$ policy optimization. 
 Our key technical contribution is a more careful treatment of the feasibility of every inner iterate, which is due to the local weak convexity of $J(\cdot)$ and nonconvexity of $\mathcal{K}$.
 Note that if $C=I_{n_x}$, 
 the weak PL condition  (\cref{theorem:weak-PL}) yields a smaller complexity order, particularly for the outer loop; see \Cref{remark:weakPL-convergence}.

\begin{remark}[Comparison with subgradient methods]\label{remark:comparison-subgradient}
For a globally weakly convex and Lipschitz function, the standard subgradient method ensures the convergence rate of $\bigO(T^{-1/4})$ for the gradient of the Moreau envelope using the stepsize $\alpha=\Theta(T^{-1/2})$ \cite{davis2019stochastic}, corresponding to an $\bigO(\epsilon^{-4})$ complexity. This guarantee does not require descent of the
original cost. Indeed, even in the deterministic setting, an explicit
subgradient update may increase the function value. Consequently, when the
objective is defined only on the nonconvex stabilizing set $\mathcal K$, the
standard analysis for subgradient methods needs to assume that every iterate remains feasible
and does not by itself provide an anytime feasibility guarantee \cite{watanabe2025policy}. In contrast,
the PBM accepts an outer update only after the acceptance  test~\eqref{eq:approx},
while \cref{lemma:feasibility_nullstep} keeps all inner iterates in
$\operatorname{int}(\widetilde{\mathcal K}_0)$. Intuitively, increasing $\rho$ plays the same role as decreasing the stepsize, as seen in \eqref{eq:L_analytic}.
Similar to the gradient descent for LQR \cite{fazel2018global,hu2023toward,watanabe2025revisiting},
appropriate control of this parameter enables us to maintain feasibility.
\hfill$\square$
\end{remark}

\begin{remark}[Convergence in the state-feedback case]\label{remark:weakPL-convergence}
When $C=I_{n_x}$, we have shown a weak PL condition in \cref{theorem:weak-PL}.
This inequality not only translates stationarity into global optimality   but also improves the complexity order to $\mathcal{O}(\epsilon^{-3})$. In particular, the outer loop achieves $J(K_T)-J^\star\leq \varepsilon$ after $T=\mathcal{O}(\varepsilon^{-1})$ iterations, while each inner-loop still requires $\mathcal{O}(\varepsilon^{-2})$.
The proof is not very difficult:
for the gap $\Delta_t=J(K_t)-\min_{L}\left\{J(L)+\frac{\rho+m}{2}\lVert L-K_t\rVert_F^2\right\}$ with $K_t\in\mathcal{K}_0$,
one can show 
\begin{equation*}
    \Delta_t \geq c
       \left(J(K_t)-J^\star\right)^2,\qquad
       c=\frac{\rho}{2
      \left(
      \ell + (m+\rho)\mu^{-1}
      \right)^2},
 \end{equation*}
 where $\mu>0$ is a constant associated with the weak PL condition on $\widetilde{\mathcal{K}}_0$.
 Combining this with the proof of \cref{lemma:outerloop} yields
 the recursion
 \begin{equation*}
     \left(J(K_{t+1})- J^\star\right)
     -\left(J(K_t)- J^\star\right)
     \leq -\beta c \left(J(K_t)-J^\star\right)^2,
 \end{equation*}
 which implies
 $  J(K_T)-J^\star \leq 
     \frac{J(K_0)-J^\star}{1+\beta c(J(K_0)-J^\star) T}=\mathcal{O}(T^{-1}).$
 Furthermore, 
 it follows from a similar argument to \cite[Proposition 10.44, Exercise 10.45]{rockafellar2009variational} that
 an $(\eta,\epsilon)$-stationary of $K\in\widetilde{\mathcal{K}}_0$ with $\eta=\Theta(\varepsilon)$
 and $\epsilon=\Theta(\varepsilon^2)$ yields $J(K)-J^\star\leq\varepsilon$,
 which implies that each inner loop
 takes $\mathcal{O}(\varepsilon^{-2})$.
 Consequently,
 under a weak PL condition on $\widetilde{\mathcal{K}}_0$
 we can achieve 
  $J(K)-J^\star\leq\varepsilon$
  after $\mathcal{O}(\varepsilon^{-3})$ iterations.
\hfill$\square$
\end{remark}

\section{Numerical experiments}\label{section:simulations}

This section presents numerical experiments for the PBM.
We first present a simple instance 
for which we can explicitly examine our theoretical results and parameters $m$, $\ell$, and $\rho$. 
We then provide further comparison with two other methods from \cite{watanabe2025policy,guo2023complexity},
showing the effectiveness of the PBM.

\subsection{A simple academic example} \label{subsection:example-1}
   Consider a single-input, single-output system with problem data:
\begin{equation*}
    A=\frac12c^\tr
c=\frac14\begin{bmatrix}1&1\\1&1\end{bmatrix},\qquad
B=c^\tr,\quad B_w=I_2,\quad C=c,
\end{equation*}
where $c=\frac{1}{\sqrt 2}\begin{bmatrix}1&1\end{bmatrix}$, 
and choose $Q=10^{-3}I_2$ and $R=10^{-2}$.  We consider the
static output-feedback control $u_t=ky_t$. 
We can write the $\mathcal{H}_\infty$ cost explicitly as
\begin{equation*}
J(k)=\frac{\sqrt{10^{-3}+10^{-2}k^2}}{1-|1/2+k|}, \quad k \in \mathcal K=(-3/2,1/2).
\end{equation*}
Note that the minimizer is uniquely $k^\star=-1/5$ with
$J^\star=\sqrt{14}/70$.

We initialize the method at $k_0=0$, for which
$J(k_0)=1/(5\sqrt{10})$. 
Then, we can directly compute 
$\mathcal K_0=
\left[1-\frac{\sqrt{21}}{3},\,0\right]$
and
$\widetilde{\mathcal K}_0=
\left[-4+\frac{\sqrt{366}}{6},\,
\frac43-\frac{\sqrt{46}}6\right]$.
It then follows that
$$
\underline{d}=\operatorname{dist}(\mathcal K_0,\partial\widetilde{\mathcal K}_0)
=\frac{4}{3}-\frac{\sqrt{46}}6>\frac{1}{5}.
$$ 
Direct differentiation on differentiable parts gives
$
J''(k)\ge-\frac1{200},\, |J'(k)|\le\frac23,
\, k\in\widetilde{\mathcal K}_0.
$ 
Further, we can verify that $\lim_{t\downarrow-1/2}J^\prime (t)- \lim_{t\uparrow-1/2}J^\prime (t) >0$.

Thus, we take
$m=0.005$ and $\ell=\frac23$.
We run the PBM with $\beta=0.9$, $\rho=8$, and
$m=1/200,\,1/20\,,1/2$.  All three choices are certified
as $\frac{\ell}{\underline{d}}+m<\frac{2\ell}{\underline{d}}-m\leq \frac{4}{3/5}-1/2<\rho=8.$
\cref{fig:sim-1_cost-decay,fig:sim-1-grad} plot the
objective gap and the exact stationarity measure against the total number
of iterations, including all inner-loop evaluations,
together with the standard subgradient method (SM)
\begin{equation}\label{subGM}
    K_{t+1} = K_t -\alpha G_t,\qquad G_t\in\partial J(K_t)
    \tag{SM}
\end{equation}
for $\alpha=0.1$.
Here, the flat regions on the plot represent the null steps. From this result, we see the convergence toward
stationarity, as well as the global optimality.
In particular, the best performance is observed for the PBM with the least conservative $m=0.005$.

\begin{figure*}
\begin{subfigure}{.48\textwidth}
  \centering
  \includegraphics[width=\columnwidth]{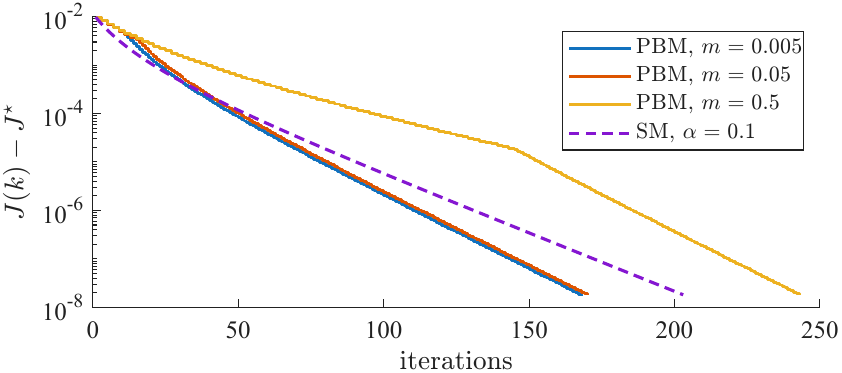}
\caption{Plot of $J(k)-J^\star$
}
    \label{fig:sim-1_cost-decay}
\end{subfigure}
\begin{subfigure}{.48\textwidth}
  \centering
\includegraphics[width=\columnwidth]{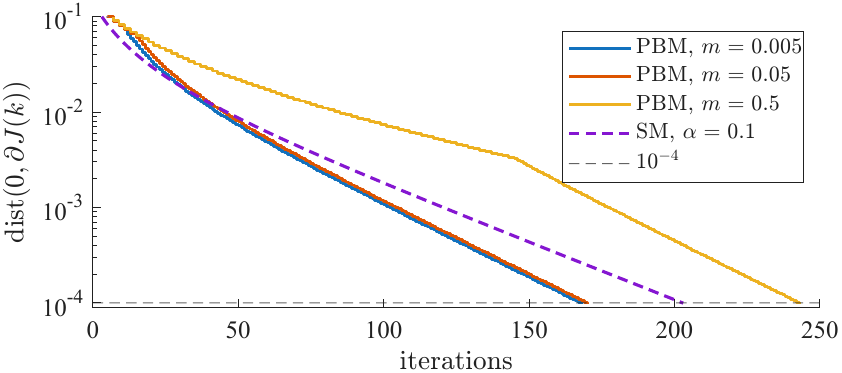}
\caption{Plot of $\mathrm{dist}(0,\partial J(k))$}
    \label{fig:sim-1-grad}
\end{subfigure}%
\caption{
The simulation results in \cref{subsection:example-1}. The iteration count includes the null steps.
}
\label{fig:hinf_plot}

\end{figure*}

\subsection{Additional experiments}
We test the empirical performance of the PBM for two different examples with partial observations, an open-loop unstable system \cite{garcia2001stabilization} (\cref{example:unstable}) and an aircraft model \cite{valadbeigi2019h} (\cref{example:aircraft}).
For comparison with the PBM, we employ the standard subgradient method in \cref{subGM},
which does not have a feasibility guarantee \cite{watanabe2025policy},
and the zeroth-order gradient method via the randomized smoothing method (RSM) in \cite[Algorithm 1]{guo2023complexity}.
We also compare these methods in the state feedback case by setting $C=I_{n_x}$, 
where the optimal value $J^\star$ can be computed by the Riccati-based MATLAB solver \texttt{hinfsyn}.

The algorithmic parameters are given as follows. For the PBM, we set $m=2$, $\rho=20$, and $\beta=0.5$, and
use the two-cut model in \cref{assumption:conditions_Jk} for constructing the lower approximation $J_k$, which admits the analytic solution \cref{eq:L_analytic} and allows for efficient computation.
For the SM, set the step size as $\alpha=0.002$. At each iteration, a subgradient is computed by using \cref{proposition:subgradient}. Finally, for the RSM \cite[Algorithm 1]{guo2023complexity}, we define the stepsize $\eta$ and the smoothing parameter $\delta$ in the same way as \cite[Appendix C]{guo2023complexity}, i.e., as $(\eta,\delta)=(0.001,0.0001)$. 

\begin{system}[Open-loop unstable example]\label{example:unstable}
As an open-loop unstable system
consider the following problem data from \cite{garcia2001stabilization}:
\begin{equation*}
    A=
    \begin{bmatrix}
0.5 & 0 & 0.2 & 1 \\
0 & -0.3 & 0 & 0.1 \\
0.01 & 0.1 & -0.5 & 0 \\
0.1 & 0 & -0.1 & -1
\end{bmatrix},\quad B=
\begin{bmatrix}
1 & 0 \\
0 & 1 \\
0 & 0 \\
-1 & 0
\end{bmatrix}, \quad C=
\begin{bmatrix}
    1 &0& 0& 1\\
    1 & 0 & 1 & 1
\end{bmatrix}.
\end{equation*}
Let $B_w=Q=I_4$ and $R= 0.1I_2$.
As the initial point, set
$$K_0 = -
\begin{bmatrix}
0.717 & 0.283\\
1.588 & -1.488
\end{bmatrix},$$
which is stabilizing \cite{garcia2001stabilization}.
We plot the simulation results in \cref{fig:JK_SOF_unstable,fig:JK_SF_unstable},
where we use $K_0C$ as the initial point in the state feedback case.
We see that the PBM exhibits the most favorable performance.
Particularly, it returns a near-optimal solution in the state feedback case, with faster convergence than the other methods.
\hfill$\square$
\end{system}

\begin{figure}
\begin{subfigure}{0.48\columnwidth}
  \centering
\includegraphics[width=0.7\columnwidth]{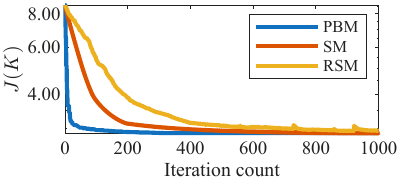}
\caption{\cref{example:unstable} $(C\neq I_{4})$
}
\label{fig:JK_SOF_unstable}
\end{subfigure}%
\begin{subfigure}{0.48\columnwidth}
  \centering
\includegraphics[width=0.7\columnwidth]{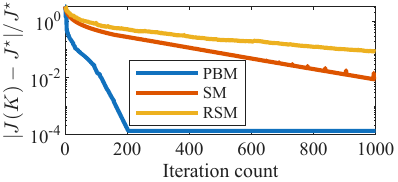}
\caption{\cref{example:unstable} ($C=I_4$)}
\label{fig:JK_SF_unstable}
\end{subfigure}
 \\
\begin{subfigure}{0.48\columnwidth}
  \centering
\includegraphics[width=0.7\columnwidth]{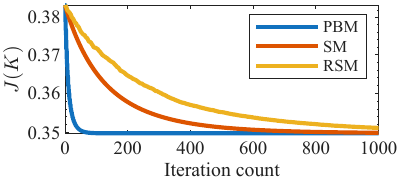}
\caption{\cref{example:aircraft} $(C\neq I_{3})$}
\label{fig:JK_SOF_ac}
\end{subfigure}%
\begin{subfigure}{0.48\columnwidth}
  \centering
\includegraphics[width=0.7\columnwidth]{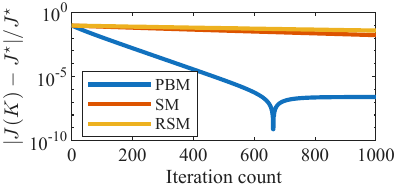}
\caption{\cref{example:aircraft} ($C=I_3$)}
\label{fig:JK_SF_ac}
\end{subfigure}
\vspace{2mm}
\caption{
Simulation results in \cref{example:aircraft,example:unstable}.
The PBM, SM, RSM stand for the proximal bundle method (\Cref{alg:PBM_outerloop}), subgradient method \cite{watanabe2025policy}, and randomized smoothing method \cite{guo2023complexity}, respectively.
Note that the iteration count includes the null steps, which correspond to the flat parts in the plots.
}
\label{fig:unstable}
\end{figure}

\begin{system}[Aircraft example]\label{example:aircraft}
Finally, consider an F-16 aircraft system given in \cite{valadbeigi2019h}:
\begin{equation*}
    \begin{aligned}
A  &=\begin{bmatrix}
0.906488 & 0.0816012 & -0.0005 \\
0.0741349 & 0.90121 & -0.000708383 \\
0 & 0 & 0.132655
\end{bmatrix},\quad 
B  =
\begin{bmatrix}
-0.00150808 \\
-0.0096 \\
0.867345
\end{bmatrix}, \\
C &=\left[\begin{array}{ccc}
0.5 & 1.6 & 0.4 \\
-1.26 & -0.9788 & 0.4852
\end{array}\right],\quad
B_w=\begin{bmatrix}
0.00951892 \\
0.00038373 \\
0 
\end{bmatrix}
\end{aligned}
\end{equation*}
with $Q=I_3$ and $R=0.1$.
In this system, \cref{assumption:stablizability} is not satisfied for $B_w$ as it is not of full row rank.
As the system is open-loop stable, we use $K_0=0$ as the initial point.
We plot the simulation results in \cref{fig:JK_SOF_ac,fig:JK_SF_ac}, where the latter shows the state feedback case $C=I_3$.
From these figures, we see that the PBM achieves the best performance not only in the convergence speed but also in accuracy.
In particular,
we obtain a near optimal solution in the state feedback case, as shown in \cref{fig:JK_SF_ac}.
Further, the PBM does not increase the function value,
while other methods, particularly the RSM, do not always generate a descent direction.
\hfill$\square$
\end{system}

\vspace{-3mm}
\section{Conclusion}\label{section:conclusion}

This paper studied policy optimization for discrete-time robust $\mathcal{H}_\infty$ control with static output-feedback. We established the weak convexity of the $\mathcal{H}_\infty$ cost on convex subsets of sublevel sets and derived an explicit characterization of its Fr\'echet subdifferential. In the state-feedback case, we further proved a weak PL inequality, which guarantees the global optimality of stationary points. Building on these properties, we developed a proximal bundle method with deterministic non-asymptotic convergence guarantees. In particular, the proposed method preserves stability for every inner- and outer-loop iterate and finds an $(\eta,\epsilon)$-stationary point within $\mathcal{O}(\max\{\eta^{-4},\epsilon^{-2}\})$ iterations. Numerical experiments illustrated the effectiveness of the proposed method. Future work includes extensions of the analysis and algorithm to broader classes of robust control problems.

\crefalias{section}{appendix}
\crefalias{subsection}{appendix}
\crefalias{subsubsection}{appendix}

\addcontentsline{toc}{section}{References}
{
\bibliographystyle{ieeetr}
\bibliography{ref.bib}

@inproceedings{tang2023global,
  title={On the Global Optimality of Direct Policy Search for Nonsmooth $\mathcal{H}_\infty$ Output-Feedback Control},
  author={Tang, Yujie and Zheng, Yang},
  booktitle={2023 62nd IEEE Conference on Decision and Control (CDC)},
  pages={6148--6153},
  year={2023},
  organization={IEEE}
}

@article{davis2019stochastic,
  title={Stochastic model-based minimization of weakly convex functions},
  author={Davis, Damek and Drusvyatskiy, Dmitriy},
  journal={SIAM Journal on Optimization},
  volume={29},
  number={1},
  pages={207--239},
  year={2019},
  publisher={SIAM}
}

@book{zhou1996robust,
  title={Robust and Optimal Control},
  author={Zhou, Kemin and Doyle, John C. and Glover, Keith},
  year={1996},
  publisher={Prentice Hall}
}

@inproceedings{fazel2018global,
  title={Global convergence of policy gradient methods for the linear quadratic regulator},
  author={Fazel, Maryam and Ge, Rong and Kakade, Sham and Mesbahi, Mehran},
  booktitle={International Conference on Machine Learning},
  pages={1467--1476},
  year={2018},
  organization={PMLR}
}

@article{burke2005robust,
  title={A robust gradient sampling algorithm for nonsmooth, nonconvex optimization},
  author={Burke, James V and Lewis, Adrian S and Overton, Michael L},
  journal={SIAM Journal on Optimization},
  volume={15},
  number={3},
  pages={751--779},
  year={2005},
  publisher={SIAM}
}

@ARTICLE{zheng2020equivalence,
  author={Y. {Zheng} and L. {Furieri} and A. {Papachristodoulou} and N. {Li} and M. {Kamgarpour}},
  journal={IEEE Transactions on Automatic Control}, 
  title={On the Equivalence of {Youla}, System-level and Input-output Parameterizations}, 
   year={2021},
  volume={66},
  number={1},
  pages={413-420},
  doi={10.1109/TAC.2020.2979785}}

@article{furieri2019input,
  title={An Input–Output Parametrization of Stabilizing Controllers: amidst {Youla} and System Level Synthesis},
  author={Furieri, Luca and Zheng, Yang and Papachristodoulou, Antonis and Kamgarpour, Maryam},
  journal={IEEE Control Systems Letters},
  volume={3},
  number={4},
  pages={1014--1019},
  year={2019},
  publisher={IEEE}
}

@InProceedings{zheng2025ECL,
  title = 	 {Extended Convex Lifting for Policy Optimization of Optimal and Robust Control},
  author =       {Zheng, Yang and Pai, Chih-Fan and Tang, Yujie},
  booktitle = 	 {Proceedings of the 7th Annual Learning for Dynamics $\&$ Control Conference},
  pages = 	 {392--404},
  year = 	 {2025},
  editor = 	 {Ozay, Necmiye and Balzano, Laura and Panagou, Dimitra and Abate, Alessandro},
  volume = 	 {283},
  series = 	 {Proceedings of Machine Learning Research},
  month = 	 {04--06 Jun},
  publisher =    {PMLR},
  url = 	 {https://proceedings.mlr.press/v283/zheng25a.html}
}

@article{zheng2022system,
  title={System-level, input-output and new parameterizations of stabilizing controllers, and their numerical computation},
  author={Zheng, Yang and Furieri, Luca and Kamgarpour, Maryam and Li, Na},
  journal={Automatica},
  volume={140},
  pages={110211},
  year={2022},
  publisher={Elsevier}
}

@article{youla1976modern,
  title={{Modern Wiener-Hopf design of optimal controllers--Part II: The multivariable case}},
  author={Youla, Dante and Jabr, Hamid and Bongiorno, Jr},
  journal={IEEE Transactions on Automatic Control},
  volume={21},
  number={3},
  pages={319--338},
  year={1976},
  publisher={IEEE}
}

@article{wang2019system,
  title={A system-level approach to controller synthesis},
  author={Wang, Yuh-Shyang and Matni, Nikolai and Doyle, John C.},
  journal={IEEE Transactions on Automatic Control},
  volume={64},
  number={10},
  pages={4079--4093},
  year={2019},
  publisher={IEEE}
}

@article{rantzer1996kalman,
  title={On the {Kalman--Yakubovich--Popov} lemma},
  author={Rantzer, Anders},
  journal={Systems \& Control Letters},
  volume={28},
  number={1},
  pages={7--10},
  year={1996},
  publisher={Elsevier}
}

@inproceedings{liao2024error,
  title={{Error bounds, PL condition, and quadratic growth for weakly convex functions, and linear convergences of proximal point methods}},
  author={Liao, Feng-Yi and Ding, Lijun and Zheng, Yang},
  booktitle={6th Annual Learning for Dynamics \& Control Conference},
  pages={993--1005},
  year={2024},
  organization={PMLR}
}

@book{boyd1994linear,
  title={Linear Matrix Inequalities in System and Control Theory},
  author={Boyd, Stephen P. and El Ghaoui,  Laurent and Feron,Eric and Balakrishnan, Venkataramanan},
  year={1994},
  publisher={SIAM}
}

@article{umenberger2022globally,
  title={Globally convergent policy search for output estimation},
  author={Umenberger, Jack and Simchowitz, Max and Perdomo, Juan and Zhang, Kaiqing and Tedrake, Russ},
  journal={Advances in Neural Information Processing Systems},
  volume={35},
  pages={22778--22790},
  year={2022}
}

@article{hu2023toward,
  title={Toward a theoretical foundation of policy optimization for learning control policies},
  author={Hu, Bin and Zhang, Kaiqing and Li, Na and Mesbahi, Mehran and Fazel, Maryam and Ba{\c{s}}ar, Tamer},
  journal={Annual Review of Control, Robotics, and Autonomous Systems},
  volume={6},
  number={1},
  pages={123--158},
  year={2023},
  publisher={Annual Reviews}
}

@inproceedings{mohammadi2019global,
  title={Global exponential convergence of gradient methods over the nonconvex landscape of the linear quadratic regulator},
  author={Mohammadi, Hesameddin and Zare, Armin and Soltanolkotabi, Mahdi and Jovanovi{\'c}, Mihailo R},
  booktitle={2019 IEEE 58th Conference on Decision and Control (CDC)},
  pages={7474--7479},
  year={2019},
  organization={IEEE}
}

@article{fatkhullin2021optimizing,
  title={Optimizing static linear feedback: Gradient method},
  author={Fatkhullin, Ilyas and Polyak, Boris},
  journal={SIAM Journal on Control and Optimization},
  volume={59},
  number={5},
  pages={3887--3911},
  year={2021},
  publisher={SIAM}
}

@article{zheng2024benign,
   author={Zheng, Yang and Pai, Chih-Fan Rich and Tang, Yujie},
  journal={IEEE Transactions on Automatic Control}, 
  title={Benign Nonconvex Landscapes in Optimal and Robust Control, Part II: Extended Convex Lifting}, 
  year={2026},
  volume={71},
  number={8},
  pages={5350-5365},
  doi={10.1109/TAC.2026.3675545}
}

@article{tang2023analysis,
  title={Analysis of the optimization landscape of Linear Quadratic {Gaussian} {(LQG)} control},
  author={Tang, Yujie and Zheng, Yang and Li, Na},
  journal={Mathematical Programming},
  volume={202},
  number={1},
  pages={399--444},
  year={2023},
  publisher={Springer}
}

@article{zheng2023benign,
  author={Zheng, Yang and Pai, Chih-Fan Rich and Tang, Yujie},
  journal={IEEE Transactions on Automatic Control}, 
  title={Benign Nonconvex Landscapes in Optimal and Robust Control, Part I: Global Optimality}, 
  year={2026},
  volume={71},
  number={8},
  pages={5334-5349},
  doi={10.1109/TAC.2026.3675544}
}

@article{talebi2024policy,
  title={Policy Optimization in Control: Geometry and Algorithmic Implications},
  author={Talebi, Shahriar and Zheng, Yang and Kraisler, Spencer and Li, Na and Mesbahi, Mehran},
  journal={arXiv preprint arXiv:2406.04243},
  year={2024}
}

@article{zhang2019policy,
  title={Policy optimization provably converges to {Nash} equilibria in zero-sum linear quadratic games},
  author={Zhang, Kaiqing and Yang, Zhuoran and Ba{\c{s}}ar, Tamer},
  journal={Advances in Neural Information Processing Systems},
  volume={32},
  year={2019}
}

@article{bu2019lqr,
  title={{LQR} through the lens of first order methods: Discrete-time case},
  author={Bu, Jingjing and Mesbahi, Afshin and Fazel, Maryam and Mesbahi, Mehran},
  journal={arXiv preprint arXiv:1907.08921},
  year={2019}
}

@article{levine1970determination,
  title={On the determination of the optimal constant output feedback gains for linear multivariable systems},
  author={Levine, William and Athans, Michael},
  journal={IEEE Transactions on Automatic Control},
  volume={15},
  number={1},
  pages={44--48},
  year={1970},
  publisher={IEEE}
}

@article{gahinet1994linear,
  title={A linear matrix inequality approach to {$H_\infty$} control},
  author={Gahinet, Pascal and Apkarian, Pierre},
  journal={International Journal of Robust and Nonlinear Control},
  volume={4},
  number={4},
  pages={421--448},
  year={1994},
  publisher={Wiley Online Library}
}

@article{zhang2020policy,
  title={Policy optimization for {$H_2$} linear control with {$H_\infty$} robustness guarantee: Implicit regularization and global convergence},
  author={Zhang, Kaiqing and Hu, Bin and Ba{\c{s}}ar, Tamer},
  journal={SIAM Journal on Control and Optimization},
  volume={59},
  number={6},
  pages={4081--4109},
  year={2021},
  publisher={SIAM}
}

@article{guo2022global,
  title={Global Convergence of Direct Policy Search for State-Feedback $\mathcal{H}_\infty$ Robust Control: A Revisit of Nonsmooth Synthesis with {Goldstein} Subdifferential},
  author={Guo, Xingang and Hu, Bin},
  journal={Advances in Neural Information Processing Systems},
  volume={35},
  pages={32801--32815},
  year={2022}
}

@article{watanabe2025revisiting,
  title={Revisiting strong duality, hidden convexity, and gradient dominance in the linear quadratic regulator},
  author={Watanabe, Yuto and Zheng, Yang},
  journal={SIAM Journal on Control and Optimization},
  volume={64},
  number={4},
  pages={2539--2568},
  year={2026},
  publisher={SIAM}
}

@article{scherer1997multiobjective,
  title={Multiobjective output-feedback control via {LMI} optimization},
  author={Scherer, Carsten and Gahinet, Pascal and Chilali, Mahmoud},
  journal={IEEE Transactions on Automatic Control},
  volume={42},
  number={7},
  pages={896--911},
  year={1997},
  publisher={IEEE}
}

@incollection{noll2013bundle,
  title={Bundle method for non-convex minimization with inexact subgradients and function values},
  author={Noll, Dominikus},
  booktitle={Computational and Analytical Mathematics: In Honor of Jonathan Borwein's 60th Birthday},
  pages={555--592},
  year={2013},
  publisher={Springer}
}

@article{apkarian2006nonsmooth,
  title={Nonsmooth {$H_\infty$} synthesis},
  author={Apkarian, Pierre and Noll, Dominikus},
  journal={IEEE Transactions on Automatic Control},
  volume={51},
  number={1},
  pages={71--86},
  year={2006},
  publisher={IEEE}
}

@article{guo2023complexity,
  title={Complexity of Derivative-Free Policy Optimization for Structured $\mathcal{H}_\infty$ Control},
  author={Guo, Xingang and Keivan, Darioush and Dullerud, Geir and Seiler, Peter and Hu, Bin},
  journal={Advances in Neural Information Processing Systems},
  volume={36},
  pages={5050--5078},
  year={2023}
}

@book{rockafellar2009variational,
  title={Variational analysis},
  author={Rockafellar, R Tyrrell and Wets, Roger J-B},
  volume={317},
  year={2009},
  publisher={Springer Science \& Business Media}
}

@article{diaz2023optimal,
  title={Optimal convergence rates for the proximal bundle method},
  author={D{\'\i}az, Mateo and Grimmer, Benjamin},
  journal={SIAM Journal on Optimization},
  volume={33},
  number={2},
  pages={424--454},
  year={2023},
  publisher={SIAM}
}

@article{liang2023proximal,
  title={Proximal bundle methods for hybrid weakly convex composite optimization problems},
  author={Liang, Jiaming and Monteiro, Renato DC and Zhang, Honghao},
  journal={arXiv preprint arXiv:2303.14896},
  year={2023}
}

@book{hiriart2013convex,
  title={Convex Analysis and Minimization Algorithms II: Advanced Theory and Bundle Methods},
  author={Hiriart-Urruty, J.-B. and Lemar{\'e}chal, C.},
  isbn={9783662064092},
  series={Grundlehren der mathematischen Wissenschaften},
  url={https://books.google.com/books?id=QinsCAAAQBAJ},
  year={2013},
  publisher={Springer Berlin Heidelberg}
}

@article{masubuchi1998lmi,
  title={{LMI}-based controller synthesis: a unified formulation and solution},
  author={Masubuchi, Izumi and Ohara, Atsumi and Suda, Nobuhide},
  journal={International Journal of Robust and Nonlinear Control: IFAC-Affiliated Journal},
  volume={8},
  number={8},
  pages={669--686},
  year={1998},
  publisher={Wiley Online Library}
}

@article{apkarian2009proximity,
  title={A proximity control algorithm to minimize nonsmooth and nonconvex semi-infinite maximum eigenvalue functions},
  author={Apkarian, Pierre and Noll, Dominikus and Prot, Olivier},
  journal={Journal of Convex Analysis},
  volume={16},
  number={3-4},
  pages={641--666},
  year={2009}
}

@article{saeki2006fixed,
  title={Fixed structure {PID} controller design for standard {$H_\infty$} control problem},
  author={Saeki, Masami},
  journal={Automatica},
  volume={42},
  number={1},
  pages={93--100},
  year={2006},
  publisher={Elsevier}
}

@article{burke2006hifoo,
  title={{HIFOO}-a {MATLAB} package for fixed-order controller design and {$\mathcal{H}_\infty$} optimization},
  author={Burke, James V and Henrion, Didier and Lewis, Adrian S and Overton, Micheal L},
  journal={IFAC Proceedings Volumes},
  volume={39},
  number={9},
  pages={339--344},
  year={2006},
  publisher={Elsevier}
}

@article{apkarian2017h,
  title={The {$H_\infty$} control problem is solved},
  author={Apkarian, Pierre and Noll, Dominikus},
  journal={Aerospace Lab},
  number={13},
  pages={1–11},
  year={2017}
}

@article{liao2025proximal,
  title={A Proximal Descent Method for Minimizing Weakly Convex Optimization},
  author={Liao, Feng-Yi and Zheng, Yang},
  journal={arXiv preprint arXiv:2509.02804},
  year={2025}
}

@inproceedings{doyle1988state,
  title={State-space solutions to standard {$H_2$} and {$H_\infty$} control problems},
  author={Doyle, John C. and Glover, Keith and Khargonekar, Pramod and Francis, Bruce},
  booktitle={1988 American Control Conference},
  pages={1691--1696},
  year={1988},
  organization={IEEE}
}

@inproceedings{watanabe2025policy,
  title={Policy optimization in robust control: Weak convexity and subgradient methods},
  author={Watanabe, Yuto and Liao, Feng-Yi and Zheng, Yang},
  booktitle={2026 American Control Conference (ACC)},
  pages={1329--1335},
  year={2026},
  organization={IEEE}
}

@article{watanabe2026gradient,
  title={Gradient Dominance in the Linear Quadratic Regulator: A Unified Analysis for Continuous-Time and Discrete-Time Systems},
  author={Watanabe, Yuto and Zheng, Yang},
  journal={arXiv preprint arXiv:2602.22577},
  year={2026}
}

@article{valadbeigi2019h,
 author={Valadbeigi, Amir Parviz and Sedigh, Ali Khaki and Lewis, F. L.},
  journal={IEEE Transactions on Neural Networks and Learning Systems}, 
  title={ {$H_{\infty}$ } Static Output-Feedback Control Design for Discrete-Time Systems Using Reinforcement Learning}, 
  year={2020},
  volume={31},
  number={2},
  pages={396-406},
  doi={10.1109/TNNLS.2019.2901889}
}

@article{garcia2001stabilization,
  title={Stabilization of discrete time linear systems by static output feedback},
  author={Garcia, Germain and Pradin, Bernard and Zeng, Fanyou},
  journal={IEEE Transactions on Automatic Control},
  volume={46},
  number={12},
  pages={1954--1958},
  year={2001},
  publisher={IEEE}
}

@book{clarke1998nonsmooth,
  title={Nonsmooth analysis and control theory},
  author={Clarke, Frank H and Ledyaev, Yu S and Stern, Ronald J and Wolenski, RR},
  year={1998},
  publisher={Springer}
}

@article{blondel1997np,
  title={{NP}-hardness of some linear control design problems},
  author={Blondel, Vincent and Tsitsiklis, John N},
  journal={SIAM Journal on Control and Optimization},
  volume={35},
  number={6},
  pages={2118--2127},
  year={1997},
  publisher={SIAM}
}

@article{lewis2007nonsmooth,
  title={Nonsmooth optimization and robust control},
  author={Lewis, Adrian S},
  journal={Annual Reviews in Control},
  volume={31},
  number={2},
  pages={167--177},
  year={2007},
  publisher={Elsevier}
}

@incollection{lemarechal1981bundle,
  title={On a bundle algorithm for nonsmooth optimization},
  author={Lemar{\'e}chal, Claude and Strodiot, Jean-Jacques and Bihain, Andr{\'e}},
  booktitle={Nonlinear programming 4},
  pages={245--282},
  year={1981},
  publisher={Elsevier}
}

@article{kiwiel1990proximity,
  title={Proximity control in bundle methods for convex nondifferentiable minimization},
  author={Kiwiel, Krzysztof C},
  journal={Mathematical programming},
  volume={46},
  number={1},
  pages={105--122},
  year={1990},
  publisher={Springer}
}

@article{hare2010redistributed,
  title={A redistributed proximal bundle method for nonconvex optimization},
  author={Hare, Warren and Sagastiz{\'a}bal, Claudia},
  journal={SIAM Journal on Optimization},
  volume={20},
  number={5},
  pages={2442--2473},
  year={2010},
  publisher={SIAM}
}

@article{kiwiel2000efficiency,
  title={Efficiency of proximal bundle methods},
  author={Kiwiel, Krzysztof C},
  journal={Journal of Optimization Theory and Applications},
  volume={104},
  number={3},
  pages={589--603},
  year={2000},
  publisher={Springer}
}

@inproceedings{furieri2020learning,
  title={Learning the globally optimal distributed {LQ} regulator},
  author={Furieri, Luca and Zheng, Yang and Kamgarpour, Maryam},
  booktitle={Learning for Dynamics and Control},
  pages={287--297},
  year={2020},
  organization={PMLR}
}

@article{kong2024cost,
  title={The cost of nonconvexity in deterministic nonsmooth optimization},
  author={Kong, Siyu and Lewis, Adrian S},
  journal={Mathematics of Operations Research},
  volume={49},
  number={4},
  pages={2385--2401},
  year={2024},
  publisher={INFORMS}
}

@article{pai2026policy,
  title={Policy Optimization of Mixed {$\mathcal{H}_2/\mathcal{H}_\infty$} Control: Benign Nonconvexity and Global Optimality},
  author={Pai, Chih-Fan and Watanabe, Yuto and Tang, Yujie and Zheng, Yang},
  journal={arXiv preprint arXiv:2603.04843},
  year={2026}
}

@article{kiwiel1995approximations,
  title={Approximations in proximal bundle methods and decomposition of convex programs},
  author={Kiwiel, Krzysztof C},
  journal={Journal of Optimization Theory and applications},
  volume={84},
  number={3},
  pages={529--548},
  year={1995},
  publisher={Springer}
}

@article{hare2009computing,
  title={Computing proximal points of nonconvex functions},
  author={Hare, Warren and Sagastiz{\'a}bal, Claudia},
  journal={Mathematical Programming},
  volume={116},
  number={1},
  pages={221--258},
  year={2009},
  publisher={Springer}
}

@article{wang2026zeroth,
  title={Zeroth-Order Nonsmooth Nonconvex Optimization with Convex Liftings and Its Application to State-Feedback {$ \mathcal{H}_\infty$} Policy Optimization},
  author={Wang, Xuhao and Tang, Yujie},
  journal={arXiv preprint arXiv:2608.23178},
  year={2026}
}

@article{li2026optimal,
  title={Optimal Deterministic Oracle Complexity for Weakly Convex Optimization},
  author={Li, Jiajin and Pan, Siyu},
  journal={arXiv preprint arXiv:2608.03246},
  year={2026}
}
}

\newpage
\appendix

\section{Technical proofs in \cref{section:analysis-2}}

\subsection{Proof of \cref{proposition:subgradient}}\label{subsection:proof-subgradient}

Now, for $t>0$
and a fixed $\theta\in\Theta$,
adding a small displacement $tdK$ yields
\begin{align*}
&
\varphi(\theta,K+tdK)
=\mathrm{Re}
\left[v^\her 
\begin{bmatrix}
Q^{1/2}\\
R^{1/2}(K+tdK)C
\end{bmatrix}(e^{j\omega }I-(A+BKC+tBdKC))^{-1}B_wu
\right] \\
=&
\mathrm{Re}
\left[v^\her 
\left(\begin{bmatrix}
Q^{1/2}\\
R^{1/2}KC
\end{bmatrix}
+
t\begin{bmatrix}
0\\
R^{1/2}dKC
\end{bmatrix}
\right)
(e^{j\omega }I-(A+BKC)-tBdKC)^{-1}B_wu
\right]\\
=&
\mathrm{Re}
\left[v^\her 
\left(\begin{bmatrix}
Q^{1/2}\\
R^{1/2}KC
\end{bmatrix}
+
t\begin{bmatrix}
0\\
R^{1/2}dKC
\end{bmatrix}
\right)
\left(I- 
t(e^{j\omega }I-(A+BKC))^{-1}BdKC
\right)^{-1}
(e^{j\omega }I-(A+BKC))^{-1}
B_wu
\right]\\
=&
\mathrm{Re}
\left[v^\her 
\left(\begin{bmatrix}
Q^{1/2}\\
R^{1/2}KC
\end{bmatrix}
+
t\begin{bmatrix}
0\\
R^{1/2}dKC
\end{bmatrix}
\right)
\left(I + t(e^{j\omega }I-(A+BKC))^{-1}BdKC
\right)
(e^{j\omega }I-(A+BKC))^{-1}
B_wu
\right]\\
&+ O(t^2) \\
=&
\varphi(\theta,K)
+ 
t\mathrm{Re}
\left[v^\her 
\begin{bmatrix}
0\\
R^{1/2}dKC
\end{bmatrix}
(e^{j\omega }I-(A+BKC))^{-1}
B_wu
\right]\\
&
+t\mathrm{Re}
\left[v^\her 
\begin{bmatrix}
Q^{1/2}\\
R^{1/2}KC
\end{bmatrix}
(e^{j\omega }I-(A+BKC))^{-1}
BdKC(e^{j\omega }I-(A+BKC))^{-1}
B_wu
\right]
+O(t^2)\\
=&
\varphi(\theta,K)
+ t
\mathrm{tr}
\left(
\mathrm{Re}\left[C(e^{j\omega }I-(A+BKC))^{-1}B_wuv^\her
\begin{bmatrix}
    0\\
    R^{1/2}
\end{bmatrix}\right]dK
\right)\\
&+t 
\mathrm{tr}
\left(
\mathrm{Re}\left[
C(e^{j\omega }I-(A+BKC))^{-1}
B_wu
v^\her 
\begin{bmatrix}
Q^{1/2}\\
R^{1/2}KC
\end{bmatrix}
(e^{j\omega }I-(A+BKC))^{-1}
B\right]dK
\right)
+ O(t^2).
\end{align*}
Therefore,
using $\Gamma_{(K,\omega)} = (e^{j\omega }I-(A+BKC))^{-1}$,
we obtain
\begin{align}\label{eq:one_subgradient}
\nabla_K \varphi(\theta,K)
= 
\mathrm{Re}&\left[
C\Gamma_{(K,\omega)} B_wuv^\her
\left(
\begin{bmatrix}
    0\\
    R^{1/2}
\end{bmatrix} 
+
\begin{bmatrix}
Q^{1/2}\\
R^{1/2}KC
\end{bmatrix}
\Gamma_{(K,\omega)} B
\right)\right]^\tr.
\end{align}
Therefore, by \cite[Theorem 10.31]{rockafellar2009variational},
\begin{align*}
\partial J(K)
=\operatorname{conv} \left\{
\mathrm{Re}\left[
C\Gamma_{(K,\omega)} B_wuv^\her
\left(
\begin{bmatrix}
    0\\
    R^{1/2}
\end{bmatrix} 
+
\begin{bmatrix}
Q^{1/2}\\
R^{1/2}KC
\end{bmatrix}
\Gamma_{(K,\omega)} B
\right)\right]^\tr
\middle|
(u,v,\omega)\in\Theta_0(K)
\right\}
\end{align*}
where
\begin{equation*}
  \Theta_0(K)
  = \{\theta\in\Theta\mid
  \varphi(\theta,K)=J(K)
  \}.
\end{equation*}

\begin{remark}[Subgradients in continuous time]
For the continuous-time systems, the $\Hinf$ cost is given by
\begin{equation*}
   J_\mathrm{c}(K)= \|\mathbf{T}_{zw}^\mathrm{c}(K)\|_{\Hinf}
    = \max_{\omega\in [0,\infty]}
    \sigma_{\max} (\mathbf{T}_{zw}^\mathrm{c}(K,\omega)),
\end{equation*}
where $\mathbf{T}_{zw}^\mathrm{c}(K,\omega)
    = 
    [(Q^{1/2})^\tr, (R^{1/2}KC)^{\tr}]^\tr
    \left(
    j\omega I- (A+BKC)
    \right)^{-1}B_w$.
This also admits an explicit subgradient thanks to the lower-$C^2$ property \cite{noll2013bundle}.
By replacing $e^{j\omega}$ by $j\omega$ in \cref{proposition:subgradient}, it is not difficult to see that 
\begin{align*}
&\partial J_\mathrm{c}(K)
=\operatorname{conv} \left\{
\mathrm{Re}\left[
C\Gamma_{(K,\omega)}^\mathrm{c} B_wuv^\her
\left(
\begin{bmatrix}
    0\\
    R^{1/2}
\end{bmatrix} 
+
\begin{bmatrix}
Q^{1/2}\\
R^{1/2}KC
\end{bmatrix}
\Gamma_{(K,\omega)}^\mathrm{c} B
\right)\right]^\tr
\middle|
(u,v,\omega)\in\Theta_0^\mathrm{c}(K)
\right\}
\end{align*}
with
$\Gamma_{(K,\omega)}^\mathrm{c} = ({j\omega}I -(A+BKC))^{-1}$ and
\begin{align}
\label{eq:Theta_K}
  &\Theta_0^\mathrm{c}(K)
  = \{\theta=(\omega,u,v)\in\Theta^\mathrm{c}\mid
  \varphi^\mathrm{c}(\theta,K)=J_\mathrm{c}(K)
  \},\\
\label{eq:f_theta_K}
  &\varphi^\mathrm{c}(\theta,K)
:= \mathrm{Re}
\left[v^\her 
\begin{bmatrix}
Q^{1/2}\\
R^{1/2}KC
\end{bmatrix}({j\omega }I-(A+BKC))^{-1}B_wu
\right],
\end{align}
where
$\Theta^\mathrm{c}=
[0,\infty]\times
\{u \in \mathbb{C}^{n_w} : \|u\|=1\} \times \{v \in \mathbb{C}^{n_x+n_u} : \|v\|=1\}$.
\hfill$\square$
\end{remark}

\subsection{Proof of \cref{lemma:weak-PL-lft}}\label{subsection:proof-weakPL-lft}

When $K$ is optimal, the inequality \cref{eq:weak-PL-lft-lemma} is obvious,
and thus we assume that $K\in\mathcal{K}$ is not optimal in the following.
Let 
$(\gamma,P)$ satisfy
$(K,\gamma,P)\in \mathcal{L}_\mathrm{lft}$.
By the convexity of $J_\mathrm{cvx}$ and
$J_\mathrm{lft}=J_\mathrm{cvx}\circ\Upsilon$, we have
\begin{align}
\label{eq:convexity_J_cvx}
\begin{aligned}
    J_\mathrm{lft}(K,\gamma,P)-J^\star=
    (J_\mathrm{cvx}\circ \Upsilon)
    (K,\gamma,P)- J^\star
    \leq& 
    \langle W,
    \Upsilon(K,\gamma,P)  
    -\Upsilon(K^\star,J^\star,P^\star)
    \rangle\\
\leq&
\|W\|_F
    \times \|\Upsilon(K,\gamma,P)    -\Upsilon(K^\star,J^\star,P^\star)\|_F,
    \end{aligned}
\end{align}
where 
$W$ is any subgradient of $J_\mathrm{cvx}$ at $\Upsilon(K,\gamma,P)$, i.e.,
$W\in \partial J_\mathrm{cvx}
    \left(
    \Upsilon(K,\gamma,P)
    \right)$, and
$(K^\star,J^\star,P^\star)$ represents an optimal triplet.

For notational simplicity, let us denote  $\Pi=\Upsilon^{-1}$, and
 $\Pi(\gamma,Y,X) = (YX^{-1},\gamma,\gamma X^{-1})$.
Now, by the chain rule 
\cite[Theorem 10.6]{rockafellar2009variational}\footnote{Note that $J_\mathrm{cvx}$ is lower semicontinuous at any $(\gamma,Y,X)\in\mathcal{F}_\mathrm{cvx}$.}
for $J_\mathrm{cvx} = J_\mathrm{lft}\circ \Pi$
with vectorization,
it holds for any $G=(G_K,G_\gamma,G_P)\in \partial J_\mathrm{lft}(K,\gamma,P)$
that
there exists a subgradient $W=(W_\gamma,W_Y,W_X)\in\partial J_\mathrm{cvx}(\Upsilon(K,\gamma,P))$ such that
\begin{equation*}
    \begin{bmatrix}
        \operatorname{vec}(W_\gamma)\\
        \operatorname{vec}(W_Y)\\
        \operatorname{vec}(W_X)
    \end{bmatrix}
    = \Gamma(\Upsilon(K,\gamma,P))^\tr 
     \begin{bmatrix}
        \operatorname{vec}(G_K)\\
        \operatorname{vec}(G_\gamma)\\
        \operatorname{vec}(G_P),
    \end{bmatrix}
\end{equation*}
where $ \Gamma(\gamma,Y,X)$ for $(\gamma,Y,X)\in\mathcal{F}_\mathrm{cvx}$ is the following matrix corresponding to the Jacobian of $\Pi$ (with suitable vectorization):
\begin{equation*}
    \Gamma(\gamma,Y,X)=
    \begin{bmatrix}
        0 & X^{-1}\otimes I_{n_u} & -(X^{-1}\otimes YX^{-1}) \\
        1 & 0 & 0\\
        \operatorname{vec}(X^{-1}) & 0 & -\gamma(X^{-1}\otimes X^{-1}).
    \end{bmatrix}
\end{equation*}
One can compute the chain rule with the Jacobian in the same manner as our prior work \cite[Lem. 6]{watanabe2026gradient}.
Note that this matrix $\Gamma$ is invertible for $(\gamma,Y,X)\in\mathcal{F}_\mathrm{cvx}$
as $\Upsilon$ is a diffeomorphism.

Consequently, invoking \cref{eq:convexity_J_cvx} 
we can set
\begin{align*}
\beta_{(K,\gamma,P)}
=
\sigma_\mathrm{max}
\left( \Gamma
\left(
\Upsilon(K,\gamma,P)
    \right)\right)
    \times
        \left\|\Upsilon(K,\gamma,P)
    -\Upsilon(K^\star,J^\star,P^\star)\right\|_F>0
\end{align*}
and obtain
\begin{align}\label{eq:weak_PL-lft}
 J_\mathrm{lft}(K,\gamma,P)-J^\star
 \leq
\|G\|_F
 \times
 \beta_{(K,\gamma,P)},
\end{align}
where
$G$ is any subgradient in 
$\partial
J_\mathrm{lft}(K,\gamma,P)$.

Note that $\beta_{(K,\gamma,P)}$ is positive for $(K,\gamma,P)\in \mathcal{L}_\mathrm{lft}$,
since $\Gamma
\left(
\Upsilon(K,\gamma,P)
\right)$ is invertible and $K\neq K^\star$.
By the arbitrariness of $G$, we obtain \cref{eq:weak-PL-lft-lemma}.

\subsection{Proof of \Cref{lemma:partial-minimization}} \label{appendix:partial-minimization}

It suffices to confirm
  that we can apply the classical result \cite[Theorem 10.13]{rockafellar2009variational}.
  Recall that $J_\mathrm{lft}(K,\gamma,P) = \gamma+\delta_{\mathcal{L}_\mathrm{lft}}(K,\gamma,P)$.
  Notice that
  $ (\gamma_K,P_K)\in \arg\min_{\gamma, P} J_\mathrm{lft}(K, \gamma, P)$ is guaranteed to exist
  from
  the non-strict bounded real lemma (\cref{proposition:BRL}).
  The constraint with the inequality \cref{eq:BRL} is closed relative to the open domain $\gamma>0$ and $P\succ0$, and 
  $J_\mathrm{lft}$ is lower semicontinuous at $(K,\gamma_K,P_K)$.
  Note that
  there is no feasible sequence approaching $P\nsucc0$ and $\gamma=0$, since \cref{eq:BRL} implies $P\succ0$ and $\gamma^2\geq\lambda_{\max}(B_w^\tr QB_w)>0$

  We next verify the level-boundedness required for \cite[Theorem 10.13]{rockafellar2009variational}.
  In particular, we show
  for any $\nu\in\mathbb{R}$, the set
\begin{equation*}
\mathcal{L}_{\mathrm{lft},\nu}
=
\{(K,\gamma,P)\mid J_{\rm lft}(K,\gamma,P)\leq \nu\}
\end{equation*}
is bounded.
While this follows from \cref{lemma:lifted_set-compactness} (\cite[Proposition 3]{wang2026zeroth}),
we prove the boundedness for completeness.
Taking the principal submatrix associated with the
first, third, fourth, and fifth block rows of \cref{eq:BLR-five-block} and applying the Schur
complement yields
\begin{equation*}
P-A_K^\tr P A_K-Q-K^\tr RK\succeq0.
\end{equation*}
Consequently,
\begin{equation*}
P\succeq Q+K^\tr RK\succeq Q\succ0.
\end{equation*}
Similarly, the principal submatrix associated with the second and
third block rows gives
\begin{equation*}
\begin{bmatrix}
\gamma^2I & B_w^\tr P\\
PB_w & P
\end{bmatrix}\succeq0,
\end{equation*}
and hence $B_w^\tr P B_w\preceq\gamma^2I.$
Since $B_w$ has full row rank and $B_wB_w^\tr \succ0$, let $B_w^\dagger:=B_w^\tr (B_wB_w^\tr )^{-1}$.
We then obtain
\begin{equation*}
P
=(B_w^\dagger)^\tr B_w^\tr P B_wB_w^\dagger
\preceq
\gamma^2(B_wB_w^\tr )^{-1}
\preceq
\nu^2(B_wB_w^\tr )^{-1}.
\end{equation*}
Thus, $P$ is uniformly bounded on $\mathcal{L}_{\mathrm{lft},\nu}$. Moreover,
using $K^\tr RK\preceq P$, we obtain
\[
\lambda_{\min}(R)\|K\|_F^2
\leq
\operatorname{tr}(P)
\leq
\nu^2\operatorname{tr}\bigl((B_wB_w^\tr )^{-1}\bigr),
\]
which uniformly bounds $K$. Since $\gamma\leq\nu$, the level set
$\mathcal{L}_{\mathrm{lft},\nu}$ is always bounded.
Consequently,
\cite[Theorem 10.13]{rockafellar2009variational} 
yields
\begin{equation}\label{eq:proof-weakPL-rockafellar}
    \partial J(K)
     \subseteq 
     \!\!\!
     \bigcup_{(\gamma_K,P_K)\in  {\arg\min}_{\gamma,P} 
     J_{\mathrm{lft}}(K,\gamma,P)
     }
     \mathcal{M}(\gamma_K,P_K),
\end{equation}
where $\mathcal{M}(\gamma,P)
    =\left\{
    G\middle|
    (G,0,0)\in
    \partial J_\mathrm{lft}(K,\gamma,P)
    \right\}.$
Therefore, we obtain the desired result.

\begin{remark}
While
\cite[Theorem 10.13]{rockafellar2009variational} 
in general ensures \cref{eq:proof-weakPL-rockafellar}
for the so-called \textit{(general) subdifferential}, not for the Fr\'echet subdifferential,
the inclusion \cref{eq:proof-weakPL-rockafellar} is valid
in
our setup for the following reason:
In the classical reference \cite[Definition 8.3]{rockafellar2009variational},
the Fr\'echet subdifferential 
is referred to as the \textit{regular subdifferential} (denoted by $\widehat{\partial}f(x)$),
and the (general) subdifferential (denoted by $\partial f(x)$) is defined as $\partial f(x) = \{v\mid \text{there are sequences $x_k\to x$ and $v_k\in \widehat{\partial}f(x_k)$ with $v_k\to v$} \}$.
While they are different and satisfy
$\widehat{\partial}f(x)\subset \partial f(x)$ in general \cite[Theorem 8.6]{rockafellar2009variational}, 
we know they coincide (i.e., $\widehat{\partial}f(x)= \partial f(x)$) when $f$ is subdifferentially regular (see \cite[Corollary 8.11]{rockafellar2009variational}).
Recalling \cref{fact:lower_C2}, the cost $J$ is subdifferentially regular, which validates \cref{eq:proof-weakPL-rockafellar}.
\hfill$\square$
\end{remark}
\section{Technical proofs in \cref{section:algorithm}}

\subsection{Proof of \cref{lemma:feasibility_nullstep}}\label{appendix:proof-feasibility-nullstep}

We prove the claim by induction.
When $k=0$, 
we have 
\begin{equation*}
    L_1=\argmin_L
    J(K_t) + \langle G_0,L-K_t\rangle
    +\frac{\rho}{2}\|L-K_t\|_F^2
    =K_t - \frac{1}{\rho}G_0,
\end{equation*}
where $G_0\in\partial J(K_t)$.
Thus, by 
$\|L_1-K_t\|_F
\leq \frac{\ell}{\rho}
<\underline{d}
$, 
where
$K_t\in \mathcal{K}_0$,
and $\underline{d}:= \min_{K\in \mathcal{K}_0}\mathrm{dist}(K,\partial \widetilde{\mathcal{K}}_0)$,
we obtain $ L_1\in\operatorname{int}(\widetilde{\mathcal{K}}_0)$.

Assuming that
$\|L_i-K_t\|_F< \underline{d}$ for $i=0,\ldots k$
for induction,
we shall show 
$\|L_{k+1}-K_t\|_F< \underline{d}$
as follows.
It follows from the optimality condition in \cref{eq:L_k+1} that
\begin{equation*}
    0\in \rho( L_{k+1}-K_t)
    + \partial J_{k}( L_{k+1}).
\end{equation*}
By \cref{assumption:conditions_Jk}, we know that
\begin{align*}
    &\partial J_{k}(L_{k+1})
    \subseteq
    \operatorname{conv}
    \{G_{0},\ldots,G_k\},\quad
    \partial J_{k}(L_{k+1})
    \subseteq
    \operatorname{conv}
    \{G_{k},S_k\},\quad\\
    \text{or}\quad&
    \partial J_{k}(L_{k+1})
    \subseteq
    \operatorname{conv}
    \{
    \{G_k\}\cup\{S_k\}\cup\partial \mathcal{M}_k(L_{k+1})
    \}
\end{align*}
for the full-memory model, two-cut model, and the more general model, respectively.
Notice that
 $\|G_i\|_F< \ell+m\underline{d} $
by $G_i-m(L_i-K_t)\in\partial J(L_i)$, $\|S_k\|_F = \rho\|K_t-L_{k}\|_F< \rho \underline{d}$ (see \cref{assumption:conditions_Jk}), 
and $\mathcal{M}_k(\cdot)$ is
$\ell_\mathcal{M}$-Lipschitz continuous.
Thus,
we obtain $ \| L_{k+1}-K_t\|_F <\underline{d}$ for all the models by
$-(L_{k+1}-K_t)\in \frac{1}{\rho}\partial J_k(L_{k+1})$ and
\begin{align*}
   & \frac{1}{\rho}\|G_i\|_F
    <
    \frac{1}{\rho}\times
     (\ell +m\underline{d})
    <\underline{d},\quad \forall i=0,\ldots,k.\\
      &\frac{1}{\rho} \|S_k\|_F 
    <
    \frac{1}{\rho}\times
    \rho \underline{d}
    \leq \underline{d},\\
   &
    \frac{1}{\rho}\|M\|_F
    \leq
    \frac{1}{\rho}
     \ell_\mathcal{M}
    <\underline{d},\quad \forall M\in \partial \mathcal{M}_k(L_{k+1}),
\end{align*} 
where we have used $\rho > \max\{m+\frac{\ell}{\underline{d}},\frac{\ell_\mathcal{M}}{\underline{d}}\}$.
By $K_t\in\mathcal{K}_0$, this yields
$ L_{k+1} \in \operatorname{int}(\widetilde{\mathcal{K}}_0)$, which completes the proof.

\subsection{Proof of \cref{lemma:inner-loop-2}}\label{appendix:proof-inner-lemma-2}

Let $d_K=K-\hat K$, $V=(m+\rho)d_K$, $U=\rho d_K$, and
\begin{equation*}
\Delta_K=J(K)-J(\hat K)-\frac{m+\rho}{2}\lVert d_K\rVert_F^2.
\end{equation*}
We have $\hat K\in\mathcal{K}_0$ and
$\|d_K\|_F=\|K-\hat K\|_F<\underline{d}$, since
\begin{equation*}
    \|d_K\|_F^2
    \leq \frac{2}{m+\rho}(J(K)-J(\hat K))
    \leq \frac{2\ell}{m+\rho}\|d_K\|
    <\underline{d}\|d_K\|_F,
\end{equation*}
where the existence of $\hat K$ is guaranteed by the coercivity (\cref{lemma:basic-properties}).
For $L$ in the largest ball centered around $K$ on which weak convexity holds,
\begin{equation*}
J(L)\ge J(\hat K)+\langle V,L-\hat K\rangle-
\frac{m}{2}\lVert L-\hat K\rVert_F^2.
\end{equation*}
Expanding around $K$ gives 
\begin{equation*}
J(L)+\frac{m}{2}\lVert L-K\rVert_F^2
\ge J(K)+\langle U,L-K\rangle-\Delta_K+
\frac{\rho}{2}\lVert d_K\rVert_F^2.
\end{equation*}
Dropping the last nonnegative term proves $U\in\partial_{\Delta_K}J(K)$ relative to this ball. 

Now, the function
\begin{equation*}
H(L)=J(L)+\frac{m+\rho}{2}\lVert L-K\rVert_F^2
\end{equation*}
is $\rho$-strongly convex on the ball and $0\in\partial H(\hat K)$. Hence
\begin{equation*}
\Delta_K=H(K)-H(\hat K)\ge\frac{\rho}{2}\lVert K-\hat K\rVert_F^2,
\end{equation*}
and therefore $\lVert U\rVert_F\le\sqrt{2\rho\Delta_K}.$
    
\subsection{Proof of 
\cref{lemma:outerloop}}\label{subsubsection:proof_outerloop}
For an outer center $K_t$, define
$h_t(L)=J(L)+\frac{m}{2}\lVert L-K_t\rVert_F^2.$
Then,
\begin{equation*}
\Delta_t=J(K_t)-\min_{L}\left\{h_t(L)+\frac{\rho}{2}\lVert L-K_t\rVert_F^2\right\}.
\end{equation*}
Let $L_{k+1}$ be the accepted inner candidate and let $J_k$ be its bundle model. Since the true proximal point lies in the local ball and $J_k\le h_t$, we have
\begin{align*}
J_k(L_{k+1})+\frac{\rho}{2}\lVert L_{k+1}-K_t\rVert_F^2
&\le J_k(\hat K_t)+\frac{\rho}{2}\lVert\hat K_t-K_t\rVert_F^2\\
&\le J(\hat K_t)+\frac{m+\rho}{2}\lVert\hat K_t-K_t\rVert_F^2\\
&=J(K_t)-\Delta_t.
\end{align*}
Therefore
\begin{equation*}
J(K_t)-J_k(L_{k+1})\ge\Delta_t.
\end{equation*}
The acceptance test gives
\begin{equation*}
J(K_t)-\left(J(L_{k+1})+\frac{m}{2}\lVert L_{k+1}-K_t\rVert_F^2\right)
\ge\beta\bigl(J(K_t)-J_k(L_{k+1})\bigr)\ge\beta\Delta_t.
\end{equation*}
With $K_{t+1}=L_{k+1}$,
\begin{equation}\label{eq:J-Delta-proof}
J(K_t)-J(K_{t+1})\ge\beta\Delta_t.
\end{equation}
Consequently, the telescoping sum for this inequality gives
\begin{equation*}
\sum_{t=0}^{T-1}\Delta_t\le\frac{J(K_0)-J^\star}{\beta}.
\end{equation*}
Set
\begin{equation*}
\sigma=\min\left\{\frac{\eta^2}{2\rho},\epsilon\right\}.
\end{equation*}
By \cref{lemma:inner-loop-2}, $\min_{t\in\{0,\ldots T-1\}}\Delta_t\le\sigma$ implies that $K_{t^*}$ 
for some
$ t^*\in\{0,\ldots T-1\}$
is an $(\eta,\epsilon)$-stationary point. If $\min_{t\in\{0,\ldots T-1\}}\Delta_t>\sigma$, we obtain
\begin{equation*}
T\le\left\lceil\frac{J(K_0)-J^\star}{\beta\sigma}\right\rceil
= 
\left\lceil
\left(\frac{J(K_0)-J^\star}{\beta}\right)
\max\left\{\frac{2\rho}{\eta^2},\frac{1}{\epsilon}\right\}
\right\rceil.
\end{equation*}

\end{document}